\documentclass{amsart}
\usepackage{graphicx} 
\usepackage{amsmath, amsthm, amssymb}
\usepackage{amsfonts}
\usepackage{amscd}
\usepackage{mathrsfs}
\usepackage{latexsym}
\usepackage{txfonts}
\usepackage{tikz}
\usepackage{graphicx}
\usepackage{subfig}
\usepackage{hyperref}
\usepackage{array}
\usepackage[all, cmtip]{xy}
\usepackage{bbm}
\usepackage{empheq}
\usepackage{color}

\newtheorem{theorem}{Theorem}[section]
\newtheorem{lemma}[theorem]{Lemma}
\newtheorem{corollary}[theorem]{Corollary}
\newtheorem{proposition}[theorem]{Proposition}

\theoremstyle{definition}
\newtheorem{definition}[theorem]{Definition}
\newtheorem{example}[theorem]{Example}
\newtheorem{remark}[theorem]{Remark}

\newcommand{\sshf}[1]{\mathcal{O}_{#1}}
\newcommand{\shf}[1]{\mathscr{#1}}

\newcommand{\prj}[1]{\mathbb{P}^{#1}}

\newcommand{\iso}{\simeq}
\newcommand{\ses}[3]{0\rightarrow#1\rightarrow#2\rightarrow#3\rightarrow{0}}

\newcommand{\rk}[1]{\textrm{rank}(#1)}
\newcommand{\is}[1]{\mathscr{I}_{#1}}

\newcommand{\paren}[1]{\left(#1\right)}

\newcommand{\bC}{{\mathbb C}}

\newcommand{\bP}{{\mathbb P}}

\newcommand{\bZ}{{\mathbb Z}}

\newcommand{\cO}{{\mathcal O}}

  \newcommand{\<}{\langle}
  \renewcommand{\>}{\rangle}

\newcommand{\suml}{\sum\limits}

\title{Mirror symmetry for certain blowups of Grassmannians}

\author{Jianxun Hu}
\address{School of Mathematics, Sun Yat-sen University, Guangzhou 510275, P.R. China}

\email{stsjxhu@mail.sysu.edu.cn}
\thanks{
 }

\author{Huazhong Ke}
\address{School of Mathematics, Sun Yat-sen University, Guangzhou 510275, P.R. China}
\email{kehuazh@mail.sysu.edu.cn}
\thanks{ 
 }

\author{Changzheng Li}
 \address{School of Mathematics, Sun Yat-sen University, Guangzhou 510275, P.R. China}
\email{lichangzh@mail.sysu.edu.cn}

\author{Lei Song}
\address{School of Mathematics, Sun Yat-sen University, Guangzhou 510275, P.R. China}
\email{songlei3@mail.sysu.edu.cn}
\thanks{ 
 }

\date{%
      }


\begin{document}
\begin{abstract}
We classify when the blowup of a complex Grassmannian $G(k, n)$ along a smooth Schubert subvariety $Z$ is Fano. We compute almost all the two-point, genus zero Gromov-Witten invariants of the blowup when $Z=G(k, n-1)$. We further prove a mirror symmetry statement for the blowup $X_{2, n}$
 of $G(2, n)$ along $G(2, n-1)$, by introducing a toric superpotential $f_{\rm tor}$ and showing the isomorphism between the Jacobi ring of $f_{\rm tor}$ and the small quantum cohomology ring $QH^*(X_{2, n})$.
\end{abstract}
\maketitle
\tableofcontents

\section{Introduction}
The blowup  $X=\rm{Bl}_ZY$ of a smooth complex projective variety $Y$ along a smooth subvariety $Z\subset Y$ is an elementary  birational  transformation in algebraic geometry. Blowup formulae of Gromov-Witten invariants play an important role to explore the structure of Gromov-Witten theory. There have   been  lots of studies of the Gromov-Witten theory of $X$, mainly in terms of blowup formulae for Gromov-Witten invariants such as
 \cite{GP98, Hu00,Gat01, Hu01,  Bay04, MaPa, HLR, Lai, HHKQ, Ke20, CLS, HKLS24,  LLW16, ChDu23, MX24}. A remarkable   decomposition of the quantum $D$-module of the blowup $X$ was proved by Iritani \cite{Iri23}. It is related to a surprising application to birational geometry (see e.g.  Kontsevich's talks \cite{Kon21} on his joint project with Katzarkov, Pantev and Yu), which draws more researchers' great interest to the Gromov-Witten theory of blowups.

Here we will study the (small) quantum cohomology ring $QH^*(X)=(H^*(X)\otimes \mathbb{C}[q_1, q_2], \star)$ of  a  concrete example $X$ of blowups in details, which is a deformation of the classical cohomology ring of $X$ by incorporating (three-point) genus-zero Gromov-Witten invariants. We were partially motivated by the exploration of  Gamma conjecture I for Fano manifolds  proposed by Galkin-Golyshev-Iritani \cite{GGI16}. The conjecture was proved for some cases including del Pezzo surfaces \cite{HKLY21} as well as toric Fano manifolds that satisfy certain condition \cite{GaIr19}. The former example is the blowup of $\mathbb{P}^2$ at points, and the latter examples include the toric blowups $\rm{Bl}_{\mathbb{P}^r}\mathbb{P}^n$ \cite{Yang}. While counterexamples of Gamma conjecture I were recently discovered in \cite{GHIKLS}, modifications of the conjecture were made and interesting connections   to birational geometry were also discussed therein.
  By \cite[Theorem 1.1]{BCW02}, it is rather restrictive for the blowup of a complex manifold at a point being Fano. It is natural to investigate the  blowup of a flag variety $Y$ (of general Lie type) along a smooth Schubert variety $Z$, in order to generalize the aforementioned examples. We were therefore led to the first case when $Y=G(k, n)=G(k, \mathbb{C}^n)$ is the complex Grassmannian parameterizing $k$-dimensional linear subspaces in $\mathbb{C}^n$. Our earlier study  \cite{HKLS24} in the special case of $Y=G(2, 4)$ provides toy examples of such blowup.

 Given a partition $\lambda=(\lambda_1, \cdots, \lambda_k) \in \mathbb{Z}^k$ with $n-k\geq \lambda_1\geq\cdots\geq \lambda_k\geq 0$,   the  Schubert subvariety  $X^\lambda$ of $G(k, n)$ is defined by   $X^\lambda=\{V\in G(k, n)\mid \dim V\cap \Lambda_{n-k+i-\lambda_i}\geq i, \, 1\leq i\leq k\}$ and has codimension $\mbox{codim} X^\lambda= \sum_i\lambda_i$. Here $\Lambda_j$   denotes the     vector subspace spanned by the first $j$ elements of the   standard basis
  $\{e_i\}_i$   of $\mathbb{C}^n$.
By \cite[Theorem 5.3]{LaWe},  the Schubert variety $X^\lambda$   is smooth if and only if the dual partition $\lambda^\vee=(n-k-\lambda_k, \cdots, n-k-\lambda_1)$ is a rectangle. When this holds,  there exist $0\le a\le k$ and $0\le b\le n-k$ such that
 \[\lambda =(\underbrace{n-k, \cdots, n-k}_a, \underbrace{n-k-b, \cdots, n-k-b}_{k-a})\]
  and   $X^\lambda=\{V\in G(k, n)\mid \Lambda_a\leq V\leq \Lambda_{k+b}\} \cong G(k-a, k+b-a)$.

  As the first main result of this paper, we show the following, the  special cases of $k=1$ or $n-1$ in which   give the above blowup $\rm{Bl}_{\mathbb{P}^r}\mathbb{P}^n$.
    \begin{theorem}\label{thm: Fano}
 Let $Z$ be a smooth Schubert subvariety of  $G(k, n)$. The blowup  $\rm{Bl}_ZG(k, n)$  is Fano if and only if
  ${\rm codim} Z\leq n$.
  \end{theorem}

Genus-zero Gromov-Witten invariants $\langle\alpha_1, \cdots, \alpha_m\rangle^X_\mathbf{d}$ of $X= \rm{Bl}_ZG(k, n)$ are given by the integration of the cup product of the pull-back of $\{\alpha_i\}\subset H^*(X)$ over the moduli space $\overline{\mathcal{M}}_{0, m}(X, \mathbf{d})$ of stable maps (with virtual fundamental class involved if needed). In particular, it vanishes unless $\mathbf{d}\in H_2(X,\bZ)$ lies in the Mori cone $\overline{\rm{NE}}(X)$ of effective curve classes of $X$.
As we will show in Proposition \ref{iso. of cone of curves}, $\overline{\rm{NE}}(X)=\mathbb{R}_{\geq 0}e+\mathbb{R}_{\geq 0}(\ell-e)$. Here
  $e$ is from the homology class of a line in the exceptional divisor $E$ of $X$, and $\ell$ is from the homology class of certain lifting of a line in $Z$ to $E$.
 We  consider  the case when   $Z=X^\lambda\cong G(k, n-1)$, for which $\lambda=(1, \cdots, 1)$, and  denote
\begin{equation}
  X_{k, n}:= \rm{Bl}_{G(k, n-1)}G(k, n).
\end{equation}
The blowup $X_{k, n}$ is special in several senses. As we will prove in \textbf{Theorem \ref{bundle over Grassmannian}}, it is the unique  blowup among $\rm{Bl}_ZG(k,n)$ (where $2\leq k\leq n-2$) that admits the structure as a projective bundle over a Grassmannian, up to the isomorphism $G(k, n)\cong G(n-k, n)$. It can be realized as a smooth Schubert variety in the two-step partial flag variety $F\ell_{k-1, k; n}$ \cite[Propositin5.4]{BC12}.  It is the only smooth Schubert variety indexed by a special partition of the form $(1, \cdots, 1, 0, \cdots, 0)$ whenever $2\leq k\leq n-2$. Moreover, the cohomology class of the blowup center $X^{(1,\cdots, 1)}$ is the generator of Seidel action on the quantum cohomology of $G(k,n)$ \cite{Sei97}.

As the second main result of this paper, we compute two-point, genus-zero Gromov-Witten invariants of $X_{k, n}$. We refer to \textbf{Theorem \ref{thm:2ptGW-precise}} for more precise statements, for ease of the notations here.
 \begin{theorem}\label{thm:2ptGW}
   For any $\alpha, \beta\in H^*(X_{k,n})$ and $\mathbf{d}\in H_2(X_{k,n}, \mathbb{Z})$, we have $\langle \alpha, \beta\rangle^{X_{k,n}}_\mathbf{d}=0$ unless $\mathbf{d}\in \{e,\ell-e, \ell\}$. Furthermore, $\langle \alpha, \beta\rangle^{X_{k,n}}_{e}$ and $\langle \alpha, \beta\rangle^{X_{k,n}}_{\ell-e}$ can be explicitly computed.
 \end{theorem}
\noindent
To obtain the above Gromov-Witten invariants, we have intentionally used different information arose from the various geometric structures of $X_{k, n}$ , including   the curve neighborhood technique \cite{BCMP13, BM15}. In this way, at least  part of the results will be potentially useful in the future study of generalizations of $X_{k, n}$ in different directions.
\begin{remark}
    It is sufficient to derive a precise ring presentation of $QH^*(X_{k, n})$ by applying Theorem \ref{thm:2ptGW}. We also computed  $\langle \alpha, \beta\rangle^{X_{2,n}}_{\ell}$ in the case of  $k=2$ with slightly more involved arguments, and expect   to obtain them for general $k$ by similar arguments. Combining with these invariants of degree $\ell$, one shall be able to obtain the quantum Chevalley formula
    for the smooth Schubert variety $X_{k, n}$. We plan to investigate such formula  elsewhere, which has its   own interest in quantum Schubert calculus.
\end{remark}

Mirror symmetry is a fascinating phenomenon arising in string theory. The (closed string)
mirror symmetry was first stated for Calabi-Yau manifolds, and  was extended to Fano manifolds $X$ by
Givental \cite{Giv95, Giv98} and Eguchi-Hori-Xiong \cite{EHX}. In this case, the mirror object is
a Landau-Ginzburg model $(\check X, f)$, consisting of a non-compact K\"ahler manifold $\check X$ and a holomorphic function $f:\check X\to \mathbb{C}$ called the superpotential.
Mirror symmetry predicts equivalences between the A-side $X$ and the B-side $(\check X, f)$ on various levels. As the first level, the (small)
quantum cohomology ring $QH^*(X)$ should be isomorphic to the Jacobi ring $\mbox{Jac}(f)$. Nevertheless, it is far from being extensively studied  beyond the cases of toric Fano varieties and flag varieties (see e.g. \cite{Cha20} and \cite{LRYZ24} and references therein).
To start  the exploration of mirror symmetry for those Fano blowups of the form $\rm{Bl}_ZG(k, n)$, we may consider $X_{k, n}$.
The partial flag variety $F\ell_{k-1, k; n}$ admits a weak Landau-Ginzburg model by using  toric degeneration \cite{BCFKS,NNU}, with the price that the resulting toric superpotential may not produce enough critical points at some specialization of the K\"ahler parameters.
We notice that $X_{k, n}$ also admits a natural  toric degeneration by restriction of that for $F\ell_{k-1, k; n}$ \cite{KM05, HLLL22}, and therefore a toric superpotential $f_{\rm tor}$ can be written down in this way. When $k=2$, we notice $X_{2, n}\cong\{V_1\leq V_2\leq \mathbb{C}^n\mid V_1\leq \Lambda_{n-1}\}$ is a $\mathbb{P}^{n-2}$-bundle over $\mathbb{P}^{n-2}$. Therefore its (quantum) cohomology is generated by divisor classes, and
 the superpotential $f_{\rm tor}$ should produce all the expected critical points.

As the third main result of this paper, we verify the above expectation and obtain the following theorem, by combining \textbf{Theorem \ref{thm: QHX2n}} and \textbf{Theorem \ref{thm:QHJac}}.
Precisely, we define the toric superpotential $f_{\rm tor}: (\mathbb{C}^*)^{2(n-2)}\times (\mathbb{C}^*)^2\to \mathbb{C}$ by
 \begin{equation}
    f_{\rm tor}(z_{i,j}, \mathbf{q}):=z_{0,1}+\sum_{j=2}^{n-2} {z_{0, j}\over z_{0, j-1}}+ {z_{1,1}\over q_1}+\sum_{j=2}^{n-2} {z_{1, j}\over z_{1, j-1}}+\sum_{j=1}^{n-2} {z_{1, j}\over z_{0, j}}+{q_1q_2\over z_{1, n-2}},
 \end{equation}
and consider the (relative) Jacobi ring
\begin{equation}
   {\rm Jac}(f_{\rm tor}):= {\mathbb{C}[z_{i,j}^{\pm 1}, q_1^{\pm 1}, q_2^{\pm 1}\mid 0\leq i\leq 1 \leq j\leq n-2]\over \big(\partial_{z_{i, j}}f_{\rm tor}\mid  0\leq i\leq 1 \leq j\leq n-2\big)}.
\end{equation}

\begin{theorem}\label{mainthm:QHJac}
  Let $ R_{\rm b}(n), R_{\sigma}(n)$ be polynomials in $\mathbb{C}[h, x, q_1, q_2]$ recursively defined by
 \begin{equation*}
    R_{\rm b}(n)=(h+x)R_{\rm b}(n-1)-(h+q_1)x R_{\rm b}(n-2),\quad \forall n\geq 2;\quad  R_{\sigma}(n)=\sum_{j=0}^{n} x^{n-j} R_{\rm b}(j),\,\,\forall n,
 \end{equation*}
 where $R_{\rm b}(0)=1, R_{\rm b}(1)=h$. Then  we have the following ring isomorphisms\footnote{Strictly speaking, a natural extension of the spectrum of the Jacobi ring has also been taken in the isomorphism \eqref{isoJac}.}
 \begin{align}
    QH^*(X_{2, n})&\cong \mathbb{C}[h, x, q_1, q_2]/(R_{\rm b}(n-1), R_{\sigma}(n-1)-q_2),\\
 \label{isoJac}    {\rm Jac}(f_{\rm tor})&\cong \mathbb{C}[h^{\pm 1}, x^{\pm 1}, q_1^{\pm 1}, q_2^{\pm 1}]/(R_{\rm b}(n-1), R_{\sigma}(n-1)-q_2).
 \end{align}
 In particular, up to a natural extension of
the spectrum of the Jacobi ring, we have
  $$ QH^*(X_{2, n}) \cong       {\rm Jac}(f_{\rm tor}).$$
\end{theorem}
\begin{remark}
    Under the above isomorphisms, $c_1(X_{2, n})=nh+(n-1)x$ is sent to the class $[f_{\rm tor}]$ in ${\rm Jac}(f_{\rm tor})$ by direct computations.
    \end{remark}
\begin{remark}
     Despite of the very strong restriction to $k=2$, the case  $X_{2, n}$ and the Schubert varieties   in $F\ell_{1, n-1; n}$ \cite{KLSong} are  important and the first toy examples in future study of mirror symmetry for smooth Schubert varieties.
We also notice that the $I$-function and quantum $D$-module of a general projective bundle was studied in \cite{IK23}, while this seems difficult to deduce precise information of $QH^*(X_{k, n})$ as we are discussing here.
\end{remark}

\begin{remark}
    As we learned from Elana Kalashnikov, the blowup $X_{2, n}$ is a $Y$-shaped quiver flag variety \cite{Kal24} with dimension vector $(1,n-2,n-2)$ and $n_{12}=n-1$, $n_{13}=1, n_{23}=1$, where $n_{ij}$ denotes  the number of arrows from vertices $i$ to $j$ of the quiver diagram. The aforementioned toric degeneration of $X_{2, n}$ obtained by restriction of that for $F\ell_{1, 2; n}$ should coincide with the generalized Gelfand-Cetlin degeneration for $Y$-shaped quiver flag varieties. There was a ring presentation of the quantum cohomology of quiver flag varieties in \cite{GK24}, together with a  mirror theorem   with respect to the superpotential proposed by Gu-Sharp \cite{GS18}. Such approach used the Abelian/non-Abelian correspondence, where  many more variables were involved and the ring presentation is quite different from the usual description even for complex Grassmannians.
    We expect that our study of the mirror symmetry for $X_{2, n}$ with respect to the Laurent polynomial superpotential $f_{\rm tor}$ will play an important role  in understanding the richer cluster structure in analogy with the case of complex Grassmannians \cite{MR20}.
\end{remark}
The paper is organized as follows. In Section 2, we characterize when the blowup ${\rm Bl}_ZG(k, n)$ is a Fano. In Section 3, we classify those   ${\rm Bl}_ZG(k, n)$ isomorphic to a projective bundle over a Grassmannian. In Section 4, we compute two-point, genus-zero Gromov-Witten invariants of  $X_{k, n}={\rm Bl}_{G(k, n-1)}G(k, n)$.
In Section 5, we provide ring presentation of the quantum cohomology of
$X_{2, n}$. In Section 6, we verify the mirror symmetry for $X_{2, n}$ by showing a ring isomorphism between $QH^*(X_{2, n})$ and the Jacobi ring of a toric superpotential $f_{\rm tor}$.

Convention: We use the symbol $\mathbb{P}(\cdot)$ to denote the space of one dimensional subspaces. So for a vector bundle $\shf{E}$ on a variety, $\mathbb{P}(\shf{E})\iso \text{Proj}\paren{\text{sym}\shf{E}^*}$.

\subsection*{Acknowledgements}

 The authors would like to thank  Rong Du, Elana Kalashnikov, Zhan Li, Leonardo Mihalcea, Heng Xie, Jinxing Xu and Weihong Xu for helpful discussions. The authors are  supported   by the National Key R \& D Program of China No.~2023YFA1009801.
H.~Ke is also supported in part by NSFC Grant 12271532.
C.~Li is also supported in part by NSFC Grant  12271529.
L.~Song is also supported in part by NSFC Grant 12471043.

\section{Fano property of the blowup $\rm{Bl}_ZG(k, n)$}
Let $Z$ be a smooth Schubert subvariety of the complex Grassmannian $$\mathbb{G}:=G(k, n)=G(k, \mathbb{C}^n)=\{V\leq \mathbb{C}^n\mid \dim V=k\}.$$
 Denote by $\mbox{codim}Z=\mbox{codim}(Z, \mathbb{G})$ the codimension of $Z$ in $\mathbb{G}$. In this section, we prove  Theorem \ref{thm: Fano}, which is restated below and characterizes of Fano property of the blowups.
 \begin{theorem}\label{Fano}
 The blowup $X={\rm  Bl}_Z\mathbb{G}$  of $\mathbb{G}$ along $Z$  is Fano if and only if $\rm{codim}Z\leq n$.
 \end{theorem}

The following consequence is immediate.
\begin{corollary}
The blowup of $\mathbb{G}$ at a point is Fano if and only if $\mathbb{G}\cong\prj{n-1}$ or $G(2, 4)$.\qed
\end{corollary}

\subsection{Smooth Schubert varieties in complex Grassmannians} Let $\{e_1, \cdots, e_n\}$ denote the standard basis of $\mathbb{C}^n$, and $\Lambda_1\leq \cdots\leq \Lambda_{n-1}\leq \Lambda_n=\mathbb{C}^n$ be the standard complete flag of $\mathbb{C}^n$, where $\Lambda_j=\mbox{Span}\{e_1,   \cdots, e_j\}$.
Let $\mathcal{P}_{k, n}:=\{(\lambda_1, \cdots, \lambda_k) \in \mathbb{Z}^k \mid n-k\geq \lambda_1\geq\cdots\geq \lambda_k\geq 0\},$ elements in which are called partitions inside $k\times (n-k)$ rectangle.

The Schubert varieties $X^\lambda$ in $\mathbb{G}$, indexed by partitions $\lambda=(\lambda_1, \cdots, \lambda_k)$ in $\mathcal{P}_{k, n}$ (and associated to $\Lambda_\bullet$), are defined by
\begin{equation}\label{SchVar}
   X^\lambda=\{V\in G(k, n)\mid \dim V\cap \Lambda_{n-k+i-\lambda_i}\geq i, \, 1\leq i\leq k\}.
\end{equation}
These are subvarieties of $\mathbb{G}$ of codimension $|\lambda|=\sum_{i=1}^k\lambda_i$.  That is, $X^\lambda$ is of dimension $|\lambda^\vee|$, where $\lambda^\vee$ denotes
 the dual partition of $\lambda$, defined by $\lambda^\vee=(n-k-\lambda_k, \cdots, n-k-\lambda_1)$.

\begin{proposition}[\protect{\cite[Theorem 5.3]{LaWe}}]
The Schubert variety $X^\lambda$ is smooth if and only if $\lambda^\vee$ is a rectangle, i.e. there exist  $0\le a\le k$, $0\le b\le n-k$ such that
\begin{equation}\label{rectanglepartition}
   \lambda =(\underbrace{n-k, \cdots, n-k}_a, \underbrace{n-k-b, \cdots, n-k-b}_{k-a}).
\end{equation}
\end{proposition}
It follows from \eqref{SchVar} that  Schubert variety $X^\lambda$ with $\lambda$ of the   form \eqref{rectanglepartition} is   given by \begin{equation}
    X^\lambda=\{V\in \mathbb{G}\mid \Lambda_a\leq V\leq \Lambda_{k+b}\} \cong G(k-a, \Lambda_{k+b}/\Lambda_{a})\cong G(k', n'),
\end{equation}
 {where } $k':=k-a,\,\, n':=k+b-a$. Throughout the rest of this paper, we always denote by  $Z=X^{\lambda}$ one such  Schubert variety, whose
 codimension is given by
 $$c:=\mbox{codim} Z=k(n-k)-k'(n'-k').$$

 The Schubert (cohomology) classes $\sigma_\lambda:=\mbox{P.D.}[X^\lambda]\in H^{2|\lambda|}(\mathbb{G}, \mathbb{Z})$ form a  $\mathbb{Z}$-additive basis of the classical cohomology $H^*(\mathbb{G}, \mathbb{Z})$. Via the natural Pl\"ucker embedding $\mathbb{G}\hookrightarrow \mathbb{P}(\wedge^k\mathbb{C}^n)$, the cut of $\mathbb{G}$ of a hyperplane of $\mathbb{P}(\wedge^k\mathbb{C}^n)$ gives an ample divisor $H$ of $\mathbb{G}$. Then  $K_{\mathbb{G}}=-nH$ is a canonical divisor of $\mathbb{G}$, and the divisor class $[H]=\sigma_{(1, 0, \cdots,0)}$ is the Schubert divisor class.

 Denote by $\pi$ the blowup $\pi: X={\rm Bl}_Z\mathbb{G}\to \mathbb{G}$, and let $E$ be the exceptional divisor.
   By abuse of notation, we simply denote the pullback  $\pi^*H$  as $H$. Moreover, we also use the same notation of a divisor for its divisor class, whenever there is no confusion.
We have the following facts,  where the isomorphism for $H^*(X)=H^*(X, \mathbb{C})$ is as vector spaces.
 \begin{equation}
    H^*(X)\cong H^*(\mathbb{G})\oplus \oplus_{i=1}^{c-1} H^*(Z), \quad  {\rm Pic}(X)=\mathbb{Z}H\oplus \mathbb{Z}E,\quad c_1(X)=-K_X=nH-(c-1)E.
 \end{equation}

\subsection{Proof of Theorem \ref{Fano}}

We start with some results on the normal bundle. Denote by $\is{Z}$ the ideal sheaf of $Z$ in $\mathbb{G}$. Note that the hyperplane section $H_Z$ of $Z$ coincides with the restriction $H|_Z$ from $\mathbb{G}$.

\begin{proposition}\label{linear section}
 The sheaf $\sshf{\mathbb{G}}(H)\otimes\is{Z}$ is globally generated. Consequently, $N^*_{Z/{\mathbb{G}}}(H)$ is globally generated.
\end{proposition}
\begin{proof}
$Z$ is cut out scheme-theoretically by some global sections of $H$ \cite[Theorem 3.11, Remark 3.12]{Ramanathan87}, i.e.
\[\oplus\sshf{\mathbb{G}}(-H)  \twoheadrightarrow \is{Z}. \]
That is,  $\sshf{\mathbb{G}}(H)\otimes\is{Z}$ is globally generated.
Tensoring the above map with $\sshf{Z}$, we get
\begin{equation}\label{eval map}
\oplus \sshf{Z}(-H)\twoheadrightarrow \is{Z}\otimes_{\sshf{\mathbb{G}}}\sshf{Z}\iso N^*_{Z/{\mathbb{G}}},
\end{equation}
which yields
$\oplus \sshf{Z}\twoheadrightarrow N^*_{Z/{\mathbb{G}}}(H).$
This finishes the proof.
\end{proof}
\begin{remark}
  The above proposition holds for any Schubert variety $Z$ by the same proof.
\end{remark}
\begin{example}
Let $\mathbb{G}=G(2, 4)$ and $Z=G(2, 3)\cong\prj{2}$. In the Plucker embedding, $\mathbb{G}$ is defined by $F=p_{12}p_{34}-p_{13}p_{24}+p_{14}p_{23}=0$. And the defining equations of $Z$ are $p_{12}=p_{13}=p_{14}=0$. We have $\sshf{\mathbb{G}}(H)\otimes\is{Z}$ is globally generated, because $\{p_{12}, p_{13}, p_{14}\}$ is a basis for $ H^0(\sshf{\mathbb{G}}(H)\otimes\is{Z})$ and a general element in the linear system vanishes only once along $Z$.


\end{example}

\begin{lemma}\label{g.g. of normal bundle}
The normal bundle $N_{Z/\mathbb{G}}$ is globally generated and
$c_1(N_{Z/\mathbb{G}})=(n-n')H_Z.$
\end{lemma}
\begin{proof}
Consider the exact sequence
\begin{equation}\label{seq for normal bundle}
    \ses{T_Z}{T_{\mathbb{G}}|_Z}{N_{Z/\mathbb{G}}}.
\end{equation}
Since $\mathbb{G}$ is a homogeneous variety, the tangent bundle $T_{\mathbb{G}}$ is globally generated, so is the quotient $N_{Z/\mathbb{G}}$. Moreover, noting $H|_Z=H_Z$,
we have
$$c_1(N_{Z/\mathbb{G}})
 =  c_1(T_{\mathbb{G}})|_Z-c_1(T_Z)
 =  nH|_Z-n'H_Z
 =  (n-n')H_Z.$$
 \end{proof}

Denote the universal subbundle (resp. quotient bundle) of $\mathbb{G}$ by $\mathcal{S}$ (resp.
$\mathcal{Q}$), and denote the universal subbundle (resp. quotient bundle) of $Z=X^\lambda$ by $\mathcal{S}_\lambda$ (resp. $\mathcal{Q}_\lambda$). Recall that $X^\lambda=\{V\in\mathbb{G}|\Lambda_a\le V\le \Lambda_{k+b}\}$. So for $V\in X^\lambda\subset\mathbb{G}$, we have
\[
\mathcal{S}_V=V,\quad\mathcal{Q}_V=\bC^n/V,\quad(\mathcal{S}_\lambda)_V=V/\Lambda_a,\quad(\mathcal{Q}_\lambda)_V=\Lambda_{k+b}/V.
\]
\begin{lemma}\label{lemma-normalbundleofsubgrassmannian}
    We have $N_{Z/\mathbb{G}}\cong\mathcal{Q}_\lambda^{\oplus a}\oplus(\mathcal{S}^\vee|_Z)^{\oplus(n-k-b)}$.
\end{lemma}
\begin{proof}
Recall that the natural identifications of vector bundles
\[
T_Z\cong \mathrm{Hom}(\mathcal{S}_\lambda,\mathcal{Q}_\lambda),\quad T_{\mathbb{G}}\cong \mathrm{Hom}(\mathcal{S},\mathcal{Q}).
\]
Note that we have the short exact sequence
\begin{align}\label{eq-sesforuniversalsubbundleonsubgrassmannian}
0\to\cO_Z^{\oplus a}\to\mathcal{S}|_Z\to\mathcal{S}_\lambda\to0,
\end{align}
which is obtained by identifying the fiber of $\cO_Z^{\oplus a}$ over $V\in Z$ with $\Lambda_a$.
 Applying $\mathrm{Hom}(\cdot,\mathcal{Q}_\lambda)$ to the short exact sequence, we get
\[
\mathrm{Hom}(\mathcal{S}_\lambda,\mathcal{Q}_\lambda)\hookrightarrow \mathrm{Hom}(\mathcal{S}|_Z,\mathcal{Q}_\lambda),\text{ and }{\mathrm{Hom}(\mathcal{S}|_Z,\mathcal{Q}_\lambda)\over \mathrm{Hom}(\mathcal{S}_\lambda,\mathcal{Q}_\lambda)}\cong \mathrm{Hom}(\cO_Z^{\oplus a},\mathcal{Q}_\lambda)\cong\mathcal{Q}_\lambda^{\oplus a}.
\]
Note that $\mathcal{Q}|_Z=\mathcal{Q}_\lambda\oplus\cO_Z^{\oplus(n-k-b)}$, where we identify the fiber of $\cO_Z^{\oplus(n-k-b)}$ over $V\in Z$ with $\mathrm{Span}\{e_{k+b+1},...,e_n\}$. Then we get
\[
\mathrm{Hom}(\mathcal{S}|_Z,\mathcal{Q}|_Z)=\mathrm{Hom}(\mathcal{S}|_Z,\mathcal{Q}_\lambda)\oplus\mathrm{Hom}(\mathcal{S}|_Z,\cO_Z^{\oplus(n-k-b)}).
\]
The inclusion $T_Z\hookrightarrow T_\mathbb{G}|_Z$ is given by the composite
\[
\mathrm{Hom}(\mathcal{S}_\lambda,\mathcal{Q}_\lambda)\hookrightarrow \mathrm{Hom}(\mathcal{S}|_Z,\mathcal{Q}|_Z)=\mathrm{Hom}(\mathcal{S}|_Z,\mathcal{Q}_\lambda)\oplus\mathrm{Hom}(\mathcal{S}|_Z,\cO_Z^{\oplus(n-k-b)}),
\]
where $\mathrm{Hom}(\mathcal{S}_\lambda,\mathcal{Q}_\lambda)$ is naturally mapped to $\mathrm{Hom}(\mathcal{S}|_Z,\mathcal{Q}_\lambda)$ via \eqref{eq-sesforuniversalsubbundleonsubgrassmannian}. So
\[
T_\mathbb{G}|_Z/T_Z\cong {\mathrm{Hom}(\mathcal{S}|_Z,\mathcal{Q}|_Z)\over \mathrm{Hom}(\mathcal{S}_\lambda,\mathcal{Q}_\lambda)}={\mathrm{Hom}(\mathcal{S}|_Z,\mathcal{Q}_\lambda)\over \mathrm{Hom}(\mathcal{S}_\lambda,\mathcal{Q}_\lambda)}\oplus\mathrm{Hom}(\mathcal{S}|_Z,\cO_Z^{\oplus(n-k-b)})\cong\mathcal{Q}_\lambda^{\oplus a}\oplus(\mathcal{S}^\vee|_Z)^{\oplus(n-k-b)}.
\]
\end{proof}

\begin{corollary}\label{splitting type}
Let $L$ be a line in $Z$. Then
$N^*_{Z/{\mathbb{G}}}|_L=\oplus^{c-(n-n')} \sshf{L}\oplus \oplus^{n-n'} \sshf{L}(-1)$.
Moreover, $c-(n-n')=0$ if and only if $k=k'=1$.
\end{corollary}
\begin{proof}
By Grothendieck, $N^*_{Z/{\mathbb{G}}}|_L=\oplus^c_{i=1} \sshf{L}(a_i)$ for some $a_i\in \mathbb{Z}$. In view of Proposition \ref{linear section}, we deduce that $a_i\ge -1$. On the other hand, by Lemma \ref{g.g. of normal bundle}, we obtain that $a_i\le 0$. So either $a_i=0$ or $-1$; among all $a_i$, number of $-1$ is $(n-n')$.

Finally consider $c-(n-n')$, the number of occurrences of $\sshf{L}$ in the splitting of $N^*_{Z/{\mathbb{G}}}|_L$. \\
(i) If $k=k'$, then
\[c-(n-n')=k(n-n')-(n-n')=(k-1)(n-n')>0,\]
unless $k=1$. \\
(ii) If $k>k'$, then
\begin{eqnarray*}
 c-(n-n')
&=&(k-1)n-(k'-1)n'-(k^2-1)+({k'}^2-1)\\
&=&(k-1)(n-k-1)-(k'-1)(n'-k'-1).
\end{eqnarray*}
We have
\[k-1\ge 1, \quad n-k-1\ge 1 \quad k'-1\ge 0, \quad n'-k'-1\ge 0\]
and
\[k-1 > k'-1, \quad n-k-1\ge n'-k'-1.\]
These inequalities imply that $c-(n-n')>0$.
\end{proof}

\begin{corollary}
Suppose $k(n-k)>2k'(n'-k')$. Then there exists a quotient $N^*\twoheadrightarrow\sshf{Z}$.
\end{corollary}
\begin{proof}
The assumption amounts to saying $\rk{N}=c=\rm{codim}(Z)>\dim Z$. Since $N$ is globally generated, by a lemma of Serre \cite[Lecture 21]{Mum66}, there exists a nowhere vanishing section $\psi\in H^0(Z, N_{Z/{\mathbb{G}}})$, defining an exact sequence
\[\ses{\sshf{Z}}{N}{N'},\]
such that $N'$ is locally free. Taking dual gives the desired quotient.
\end{proof}


Abusing notation, let $\pi$ also denote the projective bundle morphism $E=\mathbb{P}(N_{Z/{\mathbb{G}}})\rightarrow Z$. Denote by $e$ the class of a line in a fibre of $\pi$, and denote by $c_1(\sshf{E}(1))$ the first Chern class of $\sshf{E}(1)$.

Let $L$ be a line in $Z$. In view of Corollary \ref{splitting type}, there exists a quotient
\[N^*_{Z/{\mathbb{G}}}|_L\twoheadrightarrow \sshf{L},\]
which defines a closed embedding of $L$ to $E$
\[
\xymatrix{
   &  E\ar[d]_{\pi}\\
L \ar@{^{(}->}[r]\ar@{^{(}->}[ru]^{\psi} & Z.}
\]
Denote by $\ell$ the class of $\psi(L)$. We have the intersection numbers
\begin{equation}
   c_1(\sshf{E}(1))\cdot \ell=0, \: c_1(\sshf{E}(1))\cdot e=1, \:H\cdot \ell=1, \:H\cdot e=0.
\end{equation}
From these intersection numbers, we can see that $\ell$ is independent of the choice of $L$.
\begin{lemma}
  The effective cone of curves $\overline{\rm{NE}}(E)$ is spanned by the classes $e$ and $\ell-e$.
\end{lemma}
\begin{proof}
Let $C$ be an irreducible curve in $E$ with $[C]=\alpha \ell+\beta e$ for some $\alpha, \beta\in \mathbb{Z}$. Since $H$ and $\xi+H$ are nef, we have
\[\alpha=[C]\cdot H\ge 0, \alpha+\beta=[C]\cdot (\xi+H)\ge 0.\]
For any line $L$ in $Z$, in view of Corollary \ref{splitting type}, there exits a quotient
\[N^*_{Z/{\mathbb{G}}}|_L\twoheadrightarrow \sshf{L}(-1),\]
defining a section
\[
\xymatrix{
   &  E\ar[d]_{\pi}\\
L \ar@{^{(}->}[r]\ar@{^{(}->}[ru]^{\psi'} & Z.}
\]
We have $[\psi'(L)]\cdot \xi=-1$ and $[\psi'(L)]\cdot H=1$. Therefore
\[[\psi'(L)]=\ell-e.\]
We deduce that $\overline{\rm{NE}}(E)=\rm{NE}(E)$ is spanned by $e$ and $\ell-e$.
\end{proof}
The closed embedding $\iota_E: E\hookrightarrow X$ induces the map between the closed cones of curves
\[{\iota_E}_*: \overline{\rm{NE}}(E)\rightarrow \overline{\rm{NE}}(X).\]
\begin{proposition}\label{iso. of cone of curves}
$\overline{\rm{NE}}(X)={\iota_E}_*\overline{\rm{NE}}(E)$ is generated by the classes $e$ and $\ell-e$.
\end{proposition}
\begin{proof}
Let $C$ be an irreducible curve in $X$. If $C$ is contracted by $\pi$, then $[C]=me$ for some $m>0$. If not, but $\pi(C)$ is contained in $Z$, then $[C]\in {\iota_E}_*\overline{\rm{NE}}(E)$ and hence $[C]=a_1\ell+a_2e$ for some $a_1, a_2\ge 0$ subject to $a_1+a_2\ge 0$. Suppose $\pi(C)$ is not contained in $Z$, we put $[C]=a_1\ell-a_2e$ with $a_1, a_2\in \mathbb{Z}$. By the projection formula,
\[a_1=[C]\cdot H=\pi_*[C]\cdot H>0.\]
Moreover, since $H-E$ is globally generated by Proposition \ref{linear section},
\[a_1-a_2=[C]\cdot (H-E)\ge 0.\]
This finishes the proof.
\end{proof}

\begin{lemma}[Kleimann's criterion]
Let $Y$ be a smooth projective variety and $D$ an $\mathbb{R}$-divisor. Then $D$ is ample if and only if
\[\overline{\rm{NE}}(Y)\backslash\{0\}\subseteq D_{>0}.\]
\end{lemma}

\begin{proof}[Proof of Theorem \ref{Fano}]
By Kleimann's criterion and Theorem \ref{iso. of cone of curves}, $-K_X$ is ample if and only if the following hold
\begin{eqnarray*}
-K_X\cdot e &=& k(n-m)-1>0\\
-K_X\cdot (\ell-e) &=& n-[k(n-k)-k'(n'-k')-1]>0,
\end{eqnarray*}
which amounts to $k(n-k)-k'(n'-k')<n+1$, i.e. $c=\rm{codim}Z\le n$.
\end{proof}

We next describe the nef cones:
\begin{lemma}\label{nef cone}
\begin{eqnarray*}
\text{Nef}(X)&=&\mathbb{R}_{>0} H+\mathbb{R}_{>0} (H-E). \\
\text{Nef}(E) &=& \mathbb{R}_{>0} H|_E+\mathbb{R}_{>0} (H|_E+\sshf{E}(1)).
\end{eqnarray*}
\end{lemma}
\begin{proof}
In the previous section, we have seen that $H$ and $H-E$ are nef. Moreover, they are not ample, because
\[H\cdot e=0, (H-E)\cdot (\ell-e)=0.\]
For $\text{Nef}(E)$, we use Theorem \ref{iso. of cone of curves} and the fact $(H-E)|_E=H_E+\sshf{E}(1)$.
\end{proof}

\begin{corollary}
$E$ is Fano if and only if $c<n$.
\end{corollary}
\begin{proof}
By the adjunction formula,
\[K_E=(K_X+E)|_E=-nH|_E+cE|_E=-nH|_E-c\sshf{E}(1).\]
In view of Lemma \ref{nef cone},
\[-K_E=(n-c)H|_E+c(H|_E+c\sshf{E}(1)\in \text{Amp}(E),\]
if and only if $c<n$.
\end{proof}

We compute the Fano index of $X$, where $\text{gcd}$ denotes the greatest common divisor.
\begin{proposition}
Suppose $X= {\rm Bl}_Z\mathbb{G}$ is Fano. Then the Fano index is
\[r={\rm gcd}(n, \rm{codim}(Z)-1).\]
\end{proposition}
\begin{proof}
Recall $c=\rm{codim}(Z)$. Put $r_0=\rm{gcd}(n, c-1)$
\begin{eqnarray*}
-K_X&=&nH-(c-1)E\\
&=&(n-c+1)H+(c-1)(H-E)\\
&=& r_0[ \frac{(n-c+1)}{r_0}H+\frac{(c-1)}{r_0}(H-E)].
\end{eqnarray*}
Since $X$ is Fano, $n-c+1>0$. If $c>1$, then
\[\frac{(n-c+1)}{r_0}H+\frac{(c-1)}{r_0}(H-E)\]
is ample by Lemma \ref{nef cone}. If $c=1$ (which happens if and only if $k=1$ or $n-1$), then $X\iso \mathbb{G}$, the statement holds too. In any event, we are done.
\end{proof}

\section{Extremal contraction of the Fano blowup $\rm{Bl}_ZG(k, n)$}
In this section, for ease of notations, we set $e'=\ell-e$, $H'=H-E$ and  $N=\dim X=k(n-k)$. We consider the extremal contraction $\varphi$ of $e'$ by the linear system $|H'|$
\[
\xymatrix{
  X\ar[r]^{\varphi}\ar[d]^{\pi} &  X'\\
\mathbb{G} &}
\]
and classify when the morphism $\varphi$ is a projective bundle over a Grassmannian. There has been great interest to classify the varieties of Picard number two with both blowup and projective bundle structures  in the literature, e.g. \cite{Li21, BSV24}.

\subsection{Projective bundle case}
To determine whether the extremal contraction $\varphi$ is a fibration or birational morphism, one needs to consider whether the intersection number, whose vanishing is a sufficient and necessary condition for $\varphi$ being a fibration,
\begin{equation}\label{condition for fibration}
{H'}^N=H^N-\sum^{\dim{Z}}_{i=0}{N\choose i} s_{\dim{Z}-i}(N_{Z/G})\cdot H^i.
\end{equation}
Here $s_i(\cdot)$ denotes the $i$-th Segre class, cf. \cite{Fu98}.

\begin{example}
Let $X$ be the blowup of $G(2, 5)$ along $G(1, 3)$. Then by (\ref{condition for fibration}), the contraction $\varphi$ is a Fano fibration, however it is not a projective bundle over a Grassmannian, see Proposition \ref{bundle over Grassmannian}.
\end{example}

In the subsection, we shall obtain several necessary conditions for $\varphi$ being a projective bundle morphism. We assume that $X\iso \mathbb{P}(\shf{E})$ for some vector bundle $\shf{E}$ on $X'$ and $\varphi$ is the natural morphism $\mathbb{P}(\shf{E})\rightarrow X'$, and denote the tautological line bundle by $\sshf{\varphi}(1)$.


\begin{lemma}\label{surjectivity of restriction}[\cite[Lemma 2]{Li21}]
The restriction morphism $\varphi_E=\varphi|_E: E\rightarrow X'$ is surjective, and both $X$ and $X'$ are smooth Fano varieties.\qed
\end{lemma}

\begin{proposition}\label{dim of fibre}
Suppose $\varphi$ is a projective bundle. Then
\begin{itemize}
    \item[(i)] $\text{Pic}(X')=\mathbb{Z}H'$.
    \item[(ii)] $\varphi$ is a $\mathbb{P}^{n-c}$-bundle. In particular, $c<n$.
\end{itemize}
\end{proposition}
\begin{proof}
Since $\pi_*(e)\cdot H'=e\cdot (H-E)=1$, $H'$ has to be primitive, hence generates $\text{Pic}(X')$.

Since $\varphi$ is a projective bundle,
\[K_X=\varphi^*K_{X'}+\varphi^*c_1(\shf{E})-(\dim{X}-\dim{X'}+1)c_1(\sshf{\varphi}(1)).\]
It follows that
\[
-K_X\cdot e'= \dim{X}-\dim{X'}+1.
\]
On the other hand, $-K_X\cdot e'=n+1-c$.
So $\dim{X}-\dim{X'}=n-c$.
\end{proof}

\begin{proposition}
Suppose $\varphi$ is a projective bundle. Then
\begin{itemize}
    \item[(i)] $E=(a-1)H'+c_1(\sshf{\varphi}(1))$ for some $a\in \mathbb{Z}$.
    \item[(ii)] $\varphi_E: E\rightarrow X'$ is a contraction.
\end{itemize}
\end{proposition}
\begin{proof}
(i) We can uniquely put $H=aH'+bc_1(\sshf{\varphi}(1))$ for some $a, b\in \mathbb{Z}$. Since $1=H\cdot e'=0+b$, $H=aH'+c_1(\sshf{\varphi}(1))$. It follows that
$E=H-H'=(a-1)H+c_1(\sshf{\varphi}(1)).$

(ii) Thanks to (i) and the projection formula,
\[R^i\varphi_*\sshf{X}(-E)\iso R^i\varphi_*
\sshf{\varphi}(-1)\otimes \sshf{X'}(-(a-1)H')=0\] for all $i\ge 0$. Thus applying $\varphi_*$ to the exact sequence
\[\ses{\sshf{X}(-E)}{\sshf{X}}{\sshf{E}}\]
yields the desired isomorphism $\sshf{X'}\iso \varphi_*\sshf{X}\rightarrow {\varphi_E}_*\sshf{E}$. This finishes the proof.
\end{proof}

The following tables summarize the relations among generators of $\text{Pic}(X)$ and the intersection numbers.
\[\begin{bmatrix}
H \\
E
\end{bmatrix}=\begin{bmatrix}
a & 1 \\
a-1 & 1
\end{bmatrix}\begin{bmatrix}
H' \\
c_1\paren{\sshf{\varphi}(1)}
\end{bmatrix}
\quad\quad
\begin{tabular}{ | m{0.8em} | m{1em}| m{1em} | m{1em} | m{3.6em} | }
  \hline
  & $H$ & $H'$ & $E$ & $c_1\paren{\sshf{\varphi}(1)}$  \\
  \hline
  $e'$  & $1$ & $0$ & $1$ & $1$\\
  \hline
  $e$ & $0$ & $1$ & $-1$ & $-a$\\
  \hline
\end{tabular}
\]


Recall we have assumed $\varphi: X\iso \mathbb{P}(\shf{E})\rightarrow X'$ for some vector bundle $\shf{E}$. As explained in \cite[II Lemma 7.9]{Hartshorne77}, such $\shf{E}$ is not unique: one can tensor $\shf{E}$ by any line bundle $\shf{L}$ on $X'$ without changing the isomorphic type of $X$ over $X'$. Thus by tensoring $\shf{E}$ by the line bundle $\sshf{X}((a-1)H')$, we can always make
\\
\textbf{Assumption:} $a=1$, and hence $\sshf{X}(E)\iso \sshf{\varphi}(1)$.

\subsection{VMRT and projective bundle over Grassmannian}
The variety of minimal rational tangents (VMRT for short) was introduced by Hwang and Mok, and it has become a powerful tool to characterize Fano varieties of Picard number one (see e.g. the surveys \cite{Hwa12, Mok08} and references therein). In this subsection, we shall utilize the VMRT to characterize the type of image $X'$ provided it is a Grassmannian.

Fix a point $x'\in X'$. Let $\varphi^{-1}_E(x')$ be the fiber of $\varphi_E$ over $x'$ and put $Z_{x'}=\pi(\varphi^{-1}_E(x'))$. Since $H'\cdot e=1$, $\pi: \varphi^{-1}_E(x')\rightarrow Z_{x'}$ is an injection, and indeed an isomorphism.

Consider the morphism
\[\varphi_{x'}: E_{x'}:=\mathbb{P}(N_{Z/\mathbb{G}}|_{Z_{x'}})\rightarrow X',\]
which fits into the commutative diagram
$$\xymatrix{
\varphi^{-1}_E(x')\ar@{^(->}[d]\ar[rr]&&\{x'\}\ar@{^(->}[d]\\
E_{x'}\ar@/^1.4pc/[rr]|{\varphi_{x'}} \ar[d]\ar@{^(->}[r] & E\ar[d]^{\pi} \ar[r]_{\varphi_E} & X'\\
Z_{x'}\ar@{^(->}[r] & Z& }.$$

Now consider the induced morphism by $\varphi_{x'}$ from a $\mathbb{P}^{c-2}$-bundle over $Z_{x'}$ to the VMRT $\shf{C}_{x'}$ at $x'$
\[\eta: \mathbb{P}(T_{E_{x'}/Z_{x'}}|_{\varphi^{-1}_E(x')})\rightarrow \shf{C}_{x'},\]
where $T_{E_{x'}/Z_{x'}}|_{\varphi^{-1}_E(x')}$ denotes the relative tangent bundle of $E_{x'}$ over $Z_{x'}$ restricted to the subvariety $\varphi^{-1}_E(x')$.
\begin{lemma}\label{finite morphism}
$\eta$ is a finite surjective morphism.
\end{lemma}
\begin{proof}
 Since any minimal rational curve $C'$ (with respect to $H'$) passing through $x'$ represents the class $\varphi_*e$, it is the image of a rational curve $C\subset E$ with $[C]=e$, cf. Proposition \ref{iso. of cone of curves}, and hence $\eta$ is surjective.

It remains to show $\eta$ is quasi-finite. Suppose to the contrary that $\dim \eta^{-1}(v)>0$ for some $v\in \shf{C}_{x'}$, then there exists an irreducible curve $\Gamma\subset Z_{x'}$, such that for any $t\in \Gamma$, there exists a rational curve $L_t\subset \pi^{-1}(t)$ such that $L_t$ intersects $\varphi^{-1}_E(x')$ and that $\varphi_E(L_t)$ gives rise to $v$. It follows that one can find disjoint curves $\Gamma_0, \Gamma_1, \cdots, \Gamma_s\subset E$, where integer $s\ge n-n'$ is fixed, satisfying that
\begin{itemize}
\item $\Gamma_0$ is contracted by $\varphi$ to $x'$.
\item $\Gamma_i$, $1\le i\le s$ are contracted by $\varphi$ to $s$ distinct points different from $x'$.
\item $\pi(\Gamma_i)=\Gamma$, $0\le i\le s$.
\end{itemize}
Since $\pi: \varphi^{-1}_E(x')\rightarrow Z_{x'}$ is an isomorphism, it follows that $\Gamma\iso \prj{1}$ with $\Gamma\cdot H=1$. Moreover, for any $1\le i\le s$, $[\Gamma_i]=e'$, so $\pi: \Gamma_i\rightarrow \Gamma$ is an isomorphism. Thus $\Gamma_0$ and $\Gamma_i$ are distinct lifts of $\Gamma$ to $E$. On the other hand, by Corollary \ref{splitting type}, there are at most $n-n'$ lifts for $\Gamma$. This gives a contradiction.
\end{proof}

\begin{lemma}\label{dimension equality}
Suppose $X'\iso G(p, p+q)$ for some $p, q\in \mathbb{N}$. Then
\[{p+q\choose p}+{n'\choose k'}={n\choose k}. \]
\end{lemma}
\begin{proof}
By \cite[Theorem 2]{RR85}, the natural morphism $H^0(\mathbb{G}, \sshf{\mathbb{G}}(H))\rightarrow H^0(Z, \sshf{Z}(H|_Z))$ is surjective. So the standard sequence
\[\ses{\is{Z}}{\sshf{\mathbb{G}}}{\sshf{Z}},\]
yields the exact sequence
\[\ses{H^0(\mathbb{G}, \is{Z}\otimes \sshf{\mathbb{G}}(H))}{H^0(\mathbb{G}, \sshf{\mathbb{G}}(H))}{H^0(Z, \sshf{Z}(H_Z))}. \]
Finally by Proposition \ref{dim of fibre} and the projection formula,
\[{p+q\choose p}=h^0(X', \sshf{X'}(H'))=h^0(X, \sshf{X}(H-E))=h^0(X, \mathbb{G}, \is{Z}\otimes \sshf{\mathbb{G}}(H)). \]
Moreover ${n\choose k}=h^0(\mathbb{G}, \sshf{\mathbb{G}}(H))$, and
${n'\choose k'}=h^0(Z, \sshf{Z}(H_Z))$. Putting these facts together, we finish the proof.
\end{proof}

In the following, we classify all the blowups ${\rm Bl}_{Z}\mathbb{G}$ that
admit a projective bundle structure over a Grassmannian.
The projective bundle structures of the  first two situations below were contained in \cite[Proposition 5.4]{BC12}, and also appeared in  \cite[Theorem B]{BSV24}.

\begin{theorem}\label{bundle over Grassmannian}
Suppose $\varphi$ is a projective bundle over a Grassmannian. Then $\varphi$ is one of the following types (not mutually exclusive):
\begin{itemize}
    \item[(i)] a $\mathbb{P}^{n-k}$-bundle over $X'\iso G(n-k, n-1) \iso G(k-1, n-1)$, in which case $Z\iso X^{(1^k)}$.
    \item[(ii)] a $\mathbb{P}^{k}$-bundle over $X'\iso G(k, n-1)\iso G(n-k-1, n-1)$; in which case $Z\iso X^{(n-k, 0^{k-1})}$.
    \item[(iii)] when $k=1$ or $n-k=1$, a $\mathbb{P}^{n'}$-bundle over $\mathbb{P}^{n-n'-1}$.
\end{itemize}
\end{theorem}
\begin{proof}
(iii) is well-known, where $X$ is the blow up of $\mathbb{P}^{n-1}$ along a linear subspace. So we assume that $k\ge 2$ and $n-k\ge 2$.

Put $X'=G(p, p+q)$ for some $p, q\in\mathbb{N}$. For a general point $x'\in X'$, the VMRT $\shf{C}_{x'}\iso \mathbb{P}^{p-1}\times \mathbb{P}^{q-1}\hookrightarrow \mathbb{P}(T_{X', x'})\iso \mathbb{P}^{pq-1}$. By Lemma \ref{finite morphism}, $p+q=n-1$. Let $F$ be a fibre of $\mathbb{P}(T_{E_{x'}/Z_{x'}}|_{\varphi^{-1}_E(x')})\rightarrow Z_{x'}$. The induced map $\eta|_F: F\rightarrow \mathbb{P}^{p-1}\times \mathbb{P}^{q-1}$ is non-constant, otherwise by the rigidity lemma (see \cite[Lemma 1.6]{KM08}),  every fibre is contracted by $\eta$. Then $\eta$ factors through a morphism $Z_{x'}\rightarrow \mathbb{P}^{p-1}\times\mathbb{P}^{q-1}$, which cannot be surjective and hence contradicts to Lemma \ref{finite morphism}. Therefore by \cite[II. Ex. 7.3]{Hartshorne77}, we deduce that $\dim{F}=c-2\le \max\{p-1, q-2\}$.

We claim either the map $F\rightarrow \mathbb{P}^{p-1}$ or $F\rightarrow\mathbb{P}^{q-1}$ is constant. For in the contrary case,
\[\eta^*_F\paren{\sshf{\mathbb{P}^{p-1}}(1)\boxtimes \sshf{\mathbb{P}^{q-1}}(1)}\iso \sshf{F}(m)\] for some $m\ge 2$, this is impossible.

We can assume without loss of generality that $F\rightarrow\mathbb{P}^{q-1}$ is constant, thus $q-1<\dim F\le p-1$  by \cite[II. Ex. 7.3]{Hartshorne77} again. In view of the proof of Lemma \ref{finite morphism}, we deduce that $\eta|_F$ maps $F$ on to a fibre of $p_2$. Therefore $c-2=\dim F=p-1$, and hence $n-c-1=q-1$.

Now by Proposition \ref{dim of fibre}, we have
\[N=pq+(n-c)= (c-1)(n-c)+(n-c),\]
which amounts to
\[(k-c)(k+c-n)=0.\]
There are two cases:
\begin{itemize}
    \item[Case I:] $k-c=0$. Then $k(n-k-1)=(n'-k')k'$.
    There are two sub-cases:
    \begin{itemize}
        \item[case I.1:] $n-k>n'-k'$, then $n-k=n'-k'+1$ and $k=k'$.
        \item [case I.2:] $n-k=n'-k'$. By Lemma \ref{dimension equality},
      \({n-1\choose c-1}={n \choose k}-{n'\choose k'}. \)
       It follows that
       \({n-1\choose k-1}={n \choose k}-{n'\choose n-k},\)
       that is
       \[{n'\choose n-k}={n-1\choose n-k-1}.\]
       By putting $\alpha=n-k=n'-k'$ and $\beta=k-k'$ and applying Lemma \ref{combinatorics}, we conclude that this is is possible only when $k-k'=1$.
    \end{itemize}
    \item [Case II:] With analogous argument as in Case I, we obtain that $c=n-k=n'-k'$.
\end{itemize}
\end{proof}

\begin{lemma}\label{combinatorics}
Let $\alpha, \beta\in\mathbb{N}$. Then the following holds
\[
{\alpha(\beta+1)-\beta\choose \alpha}\ge {\alpha(\beta+1)-1\choose \alpha-1},
\]
and the equality occurs if and only $\alpha=1$ or $\beta=1$.
\end{lemma}
\begin{proof}
It is easy to check the equality holds when $\alpha=1$ or $\beta=1$. Now assume that $\alpha, \beta\ge 2$ and proves it is a strict inequality. We have
\begin{eqnarray*}
&&{\alpha(\beta+1)-\beta\choose \alpha}> {\alpha(\beta+1)\choose \alpha-1}\\
&\Leftrightarrow& (\alpha\beta)(\alpha\beta-1)\cdots \paren{\alpha\beta-(\beta-1)}>\alpha\paren{\alpha(\beta+1)-1}\cdots \paren{\alpha(\beta+1)-(\beta-1)}\\
&\Leftrightarrow& \beta(\beta-\frac{1}{\alpha})(\beta-\frac{2}{\alpha})\cdots (\beta-\frac{\beta-1}{\alpha})>(\beta-\frac{1}{\alpha}+1)(\beta-\frac{2}{\alpha}+1)\cdots (\beta-\frac{\beta-1}{\alpha}+1)\\
&\Leftrightarrow& \beta>\paren{1+\frac{1}{\beta-\frac{1}{\alpha}}}\paren{1+\frac{1}{\beta-\frac{2}{\alpha}}}\cdots \paren{1+\frac{1}{\beta-\frac{\beta-1}{\alpha}}}.
\end{eqnarray*}
On the other hand,
\begin{eqnarray*}
\text{RHS}&\le&  \paren{1+\frac{1}{\beta-\frac{\beta-1}{\alpha}}}^{\beta-1}
\le \paren{1+\frac{1}{\beta-\frac{\beta-1}{2}}}^{\beta-1}
= \paren{\frac{\beta+3}{\beta+1}}^{\beta-1}
< \frac{(\beta+1)^{\beta+1}}{(\beta+1)^{\beta-1}(\beta+3)},
\end{eqnarray*}
where in the second last inequality we have used the simple fact that $(\beta+3)^{\beta}<(\beta+1)^{\beta+1}$ unless $\beta=1$.
Therefore $\mbox{RHS}< \frac{(\beta+1)^2}{\beta+3}
< \beta
=\text{LHS}.$
This finishes the proof.
\end{proof}

As   another application of VMRT, we obtain that
\begin{proposition}\label{flatness}
Suppose $X'$ is a Grassmannian. Then $\varphi|_E: E\rightarrow X'$ is a projective $(n-c-1)$-bundle. Equivalently, $E$ is flat over $X'$.
\end{proposition}
\begin{proof}
Suppose $\varphi_E$ is not flat, then there exists $x'\in X'$ such that $\varphi^{-1}_E(x')=\varphi^{-1}(x')\iso \mathbb{P}^{n-c}$.
As in the beginning of this subsection, we consider the morphism
\[\varphi_{x'}: \mathbb{P}(N_{Z/\mathbb{G}}|_{Z_{x'}})\rightarrow X',\]
(but with $Z_{x'}\iso \prj{n-c}$) and the induced morphism to the VMRT
\[\eta: \mathbb{P}(T_{E_{x'}/Z_{x'}}|_{\varphi^{-1}_E(x')})\rightarrow \shf{C}_{x'}.\]
Applying Lemma \ref{finite morphism}, we deduce that $\dim \shf{C}_{x'}=n-c+(c-1)-1=n-2>n-3$. But Grassmannian is homogeneous, so its VMRT is constant at any point. This gives a contradiction. Thus $\varphi_E: E\rightarrow X'$ is flat.
\end{proof}

\section{Gromov-Witten invariants of $X_{k,n}$}\label{section-GWXkn}
In this section, we  focus on the case
\begin{equation*}
    X_{k,n}:={\rm Bl}_Z\mathbb{G}={\rm Bl}_{G(k, n-1)}G(k, n)
\end{equation*}
where $Z=X^{(1,\cdots, 1)}=\{V\in \mathbb{G}\mid V\leq \Lambda_{n-1}\}=G(k, n-1)$. Note that $X_{1, n}\cong \mathbb{P}^{n-1}$, and we always assume $1<k<n$.
As we will describe in Section \ref{section-Xkngeometry}, there are rich geometric structures on $X_{k, n}$. Consequently,  $H^*(X_{k, n})=H^*(X_{k, n},  \mathbb{Z})$ has bases $\mathcal{B}_1, \mathcal{B}_2$ arising from the blowup structure and the $\mathbb{P}^{n-k}$-bundle structure respectively, which are of the   form:
\begin{eqnarray}
\label{basisB1}\mathcal{B}_1&=&\{\sigma_\lambda\mid \lambda\in \mathcal{P}_{k, n}\}\bigcup\bigcup_{i=1}^{k-1}\{E^i\sigma_\mu\mid \mu\in \mathcal{P}_{k, n-1}\},\\
\label{basisB2}\mathcal{B}_2&=&\bigcup_{i=0}^{n-k}\{E^i\bar\sigma_\lambda \mid\lambda\in \mathcal{P}_{k-1, n-1}\}.
\end{eqnarray}
 The main result of this section is the following precise statement of Theorem \ref{thm:2ptGW},
 computing two-point, genus-zero Gromov-Witten invariants $\langle\alpha, \beta\rangle_\mathbf{d}$ of $X_{k, n}$ explicitly. We refer to Section \ref{subsection-GWef} for more details on Gromov-Witten invariants, and recall that $e, \ell-e$ are irreducible effective curve classes in $H_2(X_{k, n}, \mathbb{Z})$ that span the Mori cone $\overline{\rm NE}(X_{k, n})$
\begin{theorem}\label{thm:2ptGW-precise}
    For genus-zero two-point GW invariants of $X_{k,n}$, we have the following.
\begin{enumerate}
    \item[(a)] If $\mathbf{d}$ admits non-zero invariants,
    then $\mathbf{d}\in \{e, \ell-e, \ell\}$.
  \item[(b)] The only non-zero degree-$e$ invariants with insertions in $\mathcal{B}_1$ are
   \[\<E^{k-1}\sigma_{\mu},E^{k-1}\sigma_{\mu_+^\vee}\>^{X_{k,n}}_e=1.\]
  \item[(c)]
    The only non-zero degree-$(\ell-e)$ invariants with insertions in $\mathcal{B}_2$ are
  \[\<E^{n-k}\bar\sigma_\lambda,E^{n-k}\bar\sigma_{\lambda_-^\vee}\>^{X_{k,n}}_{\ell-e}=1.\]
\end{enumerate}
\end{theorem}
\noindent Here   $\mu_+^\vee$ (resp. $\lambda_-^\vee$) denotes the dual partition of $\mu\in \mathcal{P}_{k, n-1}$ (resp. $\lambda\in \mathcal{P}_{k-1, n-1}$). That is,
$\mu_+^\vee=(n-1-k-\mu_k, \cdots, n-1-k-\mu_1)$ and  $\lambda_-^\vee=(n-k-\lambda_{k-1}, \cdots, n-k-\lambda_1)$. We will prove part (a) of  Theorem \ref{thm:2ptGW-precise} in Section 4.3 and prove parts (b),(c) in Section 4.2.
\begin{remark}
    For degree-$\ell$ invariants, we can use blow-up formula \cite{ChDu23} to get
\[
\<\sigma_\mu,\sigma_\lambda\>^{X_{k,n}}_\ell=\<\sigma_\mu,\sigma_\lambda\>^{\mathbb{G}}_\ell,\quad\mu,\lambda\in\mathcal{P}_{k,n},
\]
and then  can use WDVV equations to determine the rest degree-$\ell$ invariants. We will discuss this elsewhere together with the quantum Chevalley formula.
\end{remark}

\subsection{Bases of  $H^*(X_{k,n})$}\label{section-Xkngeometry}
Denote $\bC_+^{n-1}:=\Lambda_{n-1}$ and $\bC_-:=\mathbb{C}\{e_n\}$. Denote
\begin{align}
	G_+:=G(k,\bC^{n-1}_+)=X^{(1, \cdots,1)}, &\qquad G_-:=\{V\leq\bC^n\mid\bC_-\leq V\}=G(k-1,\bC^n/\bC_-).
\end{align}
 {Here when $k=n-1$, we view $G_+$ as a single point.}


\noindent Recall the Pl\"{u}cker embedding
\[
G(k,n)\hookrightarrow\bP(\wedge^k\bC^n)\quad\mbox{ and }\quad G(k,\bC^{n-1}_+)\hookrightarrow\bP(\wedge^k\bC_+^{n-1}).
\]
We give a geometric  construction of  $X_{k, n}$   as a subvariety of the blowup of $\bP:=\bP(\wedge^k\bC^n)$ along $\bP_+:=\bP(\wedge^k\bC_+^{n-1})$   as follows. Let $[p_{i_1...i_k}]_{1\leq i_1<...<i_k\leq n}$ be the homogeneous coordinate on $\bP$ corresponding to the basis $(e_{i_1}\wedge\cdots\wedge e_{i_i})_{1\leq i_1<...<i_k\leq n}$ in $\wedge^k\bC^n$. Then $\bP_+$ is the linear subspace in $\bP$ defined by
$
p_{i_1...i_k}=0\mbox{ whenever }i_k=n.
$
Let $\bP_-$ be the linear subspace in $\bP$ defined by
$
p_{i_1...i_k}=0\mbox{  whenever }i_k<n.
$
Then we have $\bP_+\cap\bP_-=\emptyset$ and
\[
\dim\bP_++\dim\bP_-=\dim\bP-1.
\]
So we can consider the projection $f':\bP\dasharrow\bP_-$ from $\bP_+$, which is a rational map with the indeterminacy locus $\bP_+$. The rational map $f'$ defines a graph $\Gamma_{f'}\subset\bP\times\bP_-$ by the Zariski closure of
$
\{(y,f'(y)\mid y\in\bP\setminus\bP_+\}.
$
The natural projection  $\Gamma_{f'}\xrightarrow{p_1}\bP$ is the blowup of $\bP$ along $\bP_+$. Note that $\bP\setminus\bP_+\xrightarrow{f'}\bP_-$ is the normal bundle of $\bP_-$ in $\bP$, and the natural projection $\Gamma_{f'}\xrightarrow{p_2}\bP_-$  endows $\Gamma_{f'}$ with the projective bundle structure
\[
\Gamma_{f'}=\bP_{\bP_-}\Big(N_{\bP_-/\bP}\oplus\cO\Big),
\]
which can be viewed as the projective compactification of
\[
\bP\setminus\bP_+=N_{\bP_-/\bP}\cong\cO_{\bP_-}(1)^{\oplus{n-1\choose k}}.
\]

\subsubsection{Basis $\mathcal{B}_1$ from the blowup structure}
The Pl\"{u}cker embedding  identifies $\mathbb{G}=G(k,n)$ with a subvariety of $\bP$, and identifies $G_+$ with the intersection of $\mathbb{G}$ and $\bP_+$. Consequently, we can realize $X_{k, n}$ as $p_1^{-1}(\mathbb{G})$, and obtain the blowup
 $$\pi_1=p_1|_{X_{k, n}}: X_{k, n}\to \mathbb{G}.$$

Recall that the Schubert classes $\{\sigma_\lambda\mid\lambda\in \mathcal{P}_{k, n}\}$ form a basis of $H^*(\mathbb{G})$.
By abuse of notation, we denote the class $\pi^*\sigma_\mu$ in $H^*(X(k,n))$ by $\sigma_\mu$.
By Lemma \ref{lemma-normalbundleofsubgrassmannian},  the restriction $\pi_E:E\to G_+$  of  $\pi$ to the exceptional divisor $E$ endows $E$ with a $\bP^{k-1}$-bundle structure
\[
E=\bP_{G_+}(N_{G_+/\mathbb{G}})=\bP_{G_+}(\mathcal{S}^\vee_{G_+}).
\]
Let $\cO_E(-1)$ be the universal subbundle of $\pi_E^*N_{G_+/\mathbb{G}}=\pi_E^*\mathcal{S}^\vee_{G_+}$, and   $\mathcal{Q}_E$ be the universal quotient bundle on $E$ defined by the exact sequence
\[
0\to\cO_E(-1)\to\pi_E^*\mathcal{S}^\vee_{G_+}\to\mathcal{Q}_E\to 0.
\]
We set $\xi:=c_1(\cO_E(-1))\in H^2(E)$. Let $\{\sigma^+_{\mu}\mid\mathcal{P}_{k, n-1}\}$ be the Schubert classes for $G_+=G(k,\bC^{n-1}_+)$. Denote by $1^b$ the   partition $(1,..., 1, 0,..., 0)$ with $b$ copies of $1$. We have
\[
c_a(\mathcal{Q}_E)=\suml_{b=0}^a\pi_E^*\sigma^+_{1^{b}}\cdot(-\xi)^{a-b}, \ 0\leq a\leq k-1.
\]

Let $\iota_+:G_+\hookrightarrow \mathbb{G}$ and $\iota_E:E\hookrightarrow X_{k,n}$ be the natural inclusions. Noting
\[
\pi_E^*\sigma^+_\mu=\pi_E^*\iota_+^*\sigma_\mu=\iota_E^*\sigma_\mu,
\]
we have
\[
(\iota_E)_*\Big(\pi_E^*\sigma^+_{\mu}\cdot\xi^i\Big)=(\iota_E)_*\Big(\iota_E^*\sigma_{\mu}\Big)\cdot E^i=\sigma_\mu\cdot E^{i+1},\ i\geq0.
\]
Therefore a $\bZ$-basis of $H^*(X_{k,n})$ is given by the set
$\mathcal{B}_1$ in \eqref{basisB1}.

\begin{lemma}\label{lemma-Ek}
	$E^k=\suml_{i=0}^{k-1}(-1)^{k-1-i}E^i\sigma_{1^{k-i}}$.
\end{lemma}
\begin{proof}
	It follows from  \cite[Proposition 6.7]{Fu98} that
	\begin{align*}
	&(\iota_+)_*[G_+]=(\iota_E)_*(c_{k-1}(\mathcal{Q}_E))=\suml_{b=0}^{k-1}(-1)^{k-1-b}E^{k-b}\sigma_{1^b}\\
	\Rightarrow&E^k=\suml_{b=1}^{k-1}(-1)^{b+1}E^{k-b}\sigma_{1^b}+(-1)^{k+1}(\iota_+)_*[G_+].
	\end{align*}
	
	Since $G_+=X^{1^k}$, it follows that $(\iota_+)_*[G_+]=\sigma_{1^k}$. Then we are done.
\end{proof}
As a consequence, the dual basis $\mathcal{B}_1^\vee$ of $\mathcal{B}_1$ in $H^*(X(k,n),\bZ)$ with respect to the Poincar\'e pairing is given by
\begin{equation}
   \mathcal{B}_1^\vee=\{(\sigma_\lambda)^\vee\mid \lambda\in \mathcal{P}_{k, n}\}\bigcup\bigcup_{i=1}^{k-1}\{(E^i\sigma_\mu)^\vee\mid \mu\in  \mathcal{P}_{k, n-1}\}
\end{equation}
where we recall $\mu_+^\vee:=(n-1-k-\mu_k,...,n-1-k-\mu_1)$ and have
\begin{align}
(\sigma_\lambda)^\vee&=\sigma_{\lambda^\vee},\qquad	(E^i\sigma_\mu)^\vee=\sigma_{\mu_+^\vee}\cdot\suml_{a=1}^{k-i}(-1)^{a+(i-1)}E^a\sigma_{1^{k-i-a}}.
\end{align}


\subsubsection{Basis $\mathcal{B}_2$ from the $\mathbb{P}^{n-k}$-bundle structure}
The restriction $\pi_2:=p_2|_{X_{k,n}}$
endows $X_{k,n}$ with a $\mathbb{P}^{n-k}$-bundle structure. Below we provide an explicit description of the fiber and the base of this bundle. We notice that the intersection of $\mathbb{G}$ and $\bP_-$ is $G_-$. The following lemma shows that the image of the restriction of $f'$ to $\mathbb{G}$ is contained in $G_-$.

\begin{lemma}\label{lemma-imageofrestrctionoff'}
	For $V\in \mathbb{G}\setminus G_+$, we have $f'(V)=(V\cap\bC_+^{n-1})+\bC_-\in G_-$.
\end{lemma}
\begin{proof}
	Let $[p_{i_1...i_k}(V)]_{1\leq i_1<...<i_k\leq n}$ be a homogeneous coordinate of $V$. Note that the coordinate of a point in the smallest linear subspace in $\bP$ containing both $\bP_+$ and $V$ has the form
	\[
	p_{i_1...i_k}=\left\{\begin{array}{cc}\lambda_{i_1...i_k}+z p_{i_1...i_k}(V),&1\leq i_1<...<i_k<n,\\
	\\
	z p_{i_1...i_k}(V),&1\leq i_1<...<i_k=n,\end{array}\right.
	\]
	where $\lambda_{i_1...i_k},z\in\bC$. So we see that the coordinate of $f'(V)$ has the form
	\[
	p'_{i_1...i_k}=\left\{\begin{array}{cc}0,&1\leq i_1<...<i_k<n,\\
	\\
	p_{i_1...i_k}(V),&1\leq i_1<...<i_k=n.\end{array}\right.
	\]
	Note that we can choose a basis of $V$ to have the form in row vectors
	\[
	\begin{bmatrix}
		a_{11}&...&a_{1,n-1}&0\\
		\vdots&&\vdots&\vdots\\
		a_{k-1,1}&...&a_{k-1,n-1}&0\\
		a_{k1}&...&a_{k,n-1}&1
	\end{bmatrix},
	\]
	which gives a basis of $(V\cap\bC^{n-1}_+)+\bC_-$ with the form
	\[
	\begin{bmatrix}
		a_{11}&...&a_{1,n-1}&0\\
		\vdots&&\vdots&\vdots\\
		a_{k-1,1}&...&a_{k-1,n-1}&0\\
		0&...&0&1
	\end{bmatrix}.
	\]
	Thus the coordinate of $(V\cap\bC^{n-1}_+)+\bC_-$ is  given by $p'_{i_1...i_k}$. This proves the lemma.
\end{proof}

Let $f:\mathbb{G}\dasharrow G_-$ be the restriction of the above mentioned $f'$. Then $f$ is a rational map with the indeterminacy locus $G_+$. The graph of $f$ is the strict transform of $\mathbb{G}\subset\bP$ under the blowup $p_1$, so that its closure is just $X_{k, n}=p_1^{-1}(\mathbb{G})$. Note that Lemma \ref{lemma-imageofrestrctionoff'} implies that, for $W\in G_-$, $f^{-1}(W)$ consists of $V\in\mathbb{G}\setminus G_+$ satisfying $V\cap\bC_+^{n-1}=W\cap\bC_+^{n-1}$. One can deduce from this that $\mathbb{G}\setminus G_+\xrightarrow{f}G_-$ is the normal bundle of $G_-$ in $\mathbb{G}$. As a consequence, $
\pi_2:X_{k,n}\to G_-$ endows $X_{k,n}$ with the $\bP^{n-k}$-bundle structure
\begin{equation}\label{XknBundle}
    X_{k,n}=\bP_{G_-}(N_{G_-/\mathbb{G}}\oplus\cO),
\end{equation}
which, by Lemma \ref{lemma-normalbundleofsubgrassmannian}, can be viewed as the projective compactification of
\[
G(k,n)\setminus G_+=N_{G_-/\mathbb{G}}\cong\mathcal{Q}_{G_-}.
\]
Moreover, the exceptional divisor $E$   is exactly the divisor at infinity of the $\bP^{n-k}$-bundle, i.e.
\begin{equation}
    E=\bP_{G_-}(\mathcal{Q}_{G_-}\oplus\{0\})\subset\bP_{G_-}(\mathcal{Q}_{G_-}\oplus\cO)=X_{k,n}.
\end{equation}

Let $\{\bar\sigma_\lambda\mid\lambda\in \mathcal{P}_{k-1, n-1}\}$ be the Schubert classes for $G_-$.
Let $\cO_{X(k,n)}(-1)$ be the universal subbundle of $\pi_2^*(\mathcal{Q}_{G_-}\oplus\cO)$, and let $\mathcal{Q}_{X_{k,n}}$ be the universal quotient bundle on $X_{k,n}$ defined by the exact sequence
\[
0\to \cO_{X_{k,n}}(-1)\to\pi_2^*(\mathcal{Q}_{G_-}\oplus\cO)\to\mathcal{Q}_{X_{k,n}}\to0.
\]
Note   $E=c_1(\cO_{X_{k,n}}(1))$ and simply denote the special partition $(j,0,\cdots, 0)$ as $j$. We have
\[
c_j(\mathcal{Q}_{X_{k,n}})=\sum_{i=0}^jE^i\bar\sigma_{j-i}, \mbox{ implying } H^*(X_{k,n})=H^*(G_-)[E]/(E^{n-k+1}+\suml_{a=1}^{n-k}E^{a}\bar\sigma_{n-k+1-a}).
\]
 Hence, a $\bZ$-basis of $H^*(X_{k,n})$ is given by the set $\mathcal{B}_2$ in \eqref{basisB2}.
By \cite[Lemma 2.4]{LLW16}, the dual basis of $\mathcal{B}_2$ in $H^*(X(k,n))$ with respect to the Poincar\'e pairing is given by
\[\mathcal{B}_2^\vee=\bigcup_{i=0}^{n-k}\{(E^i\bar\sigma_\lambda)^\vee \mid \lambda\in \mathcal{P}_{k-1, n-1}\}, \mbox{ where }
(E^i\bar\sigma_\lambda)^\vee=\bar\sigma_{\lambda_-^\vee}\cdot c_{n-k-i}(\mathcal{Q}_{X_{k,n}})=\bar\sigma_{\lambda_-^\vee}\cdot\suml_{a=0}^{n-k-i}E^a\bar\sigma_{n-k-i-a},
\]
where we recall $\lambda_-^\vee=(n-k-\lambda_{k-1},...,n-k-\lambda_1)$.


\subsection{Gromov-Witten invariants of degrees $e$ and $\ell-e$} \label{subsection-GWef}
For a smooth projective variety $X$, denote by $\overline{\mathcal{M}}_{0, m}(X, \mathbf{d})$
the moduli space of stable maps to $X$ (see e.g. \cite{FuPa}). It consists of  (equivalence classes of) stable maps $(f: C\to X; {\rm pt}_1,\cdots, {\rm pt}_m)$  of degree $\mathbf{d}\in H_2(X, \mathbb{Z})$, where $C$ is a tree of $\mathbb{P}^1$'s and ${\rm pt}_i$ are non-singular  points in $C$.
Let ${\rm ev}_i:\overline{\mathcal{M}}_{0, m}(X, \mathbf{d})\to X$ denote the $i$th   evaluation map, which sends $(f: C\to X; {\rm pt}_1,\cdots, {\rm pt}_m)$ to $f({\rm pt}_i)$.   For $\gamma_1, \cdots, \gamma_m\in H^*(X)$, a genus-zero,
$m$-point Gromov-Witten
invariant is defined by
\begin{equation}
    \langle \gamma_1, \cdots, \gamma_m\rangle^X_{\mathbf{d}}:=\int_{[\overline{\mathcal{M}}_{0,m}(X, \mathbf{d})]^{\rm vir}} {\rm ev}_1^*(\gamma_1)\cup\cdots \cup {\rm ev}_m^*(\gamma_m).
\end{equation}
Here $[\overline{\mathcal{M}}_{0,m}(X, \mathbf{d})]^{\rm vir}\in H_{2{\rm expdim}}(\overline{\mathcal{M}}_{0,m}(X, \mathbf{d}), \mathbb{Q})$ is the virtual fundamental class with
\begin{equation}
    {\rm expdim}= \dim X+\int_{\mathbf{d}}c_1(X)+m-3.
\end{equation}
We have
$     \langle \gamma_1, \cdots, \gamma_m\rangle^X_{\mathbf{d}}=0$   unless $\mathbf{d}\in \overline{\rm NE}(X)$. Moreover, we have
 \begin{equation}\label{divisorAxiom}
     \langle \gamma_1, \cdots, \gamma_m\rangle^X_{\mathbf{d}}=\begin{cases}
         \int_{[X]} \gamma_1\cup\cdots\cup \gamma_m,&\mbox{if } \mathbf{d}=0\text{ and }m=3,\\
         \langle \gamma_2, \cdots, \gamma_m\rangle^X_{\mathbf{d}}\cdot \int_{\mathbf{d}}\gamma_1,&\mbox{if } \gamma_1\in H^2(X),\text{ and either }\mathbf{d}\neq0\text{ or }m\ge 4.
     \end{cases}
\end{equation}

\bigskip

\begin{proof}[Proof of Theorem \ref{thm:2ptGW-precise} (b)]
    The arguments are similar to that for \cite[Lemma 3.7]{HKLS24}. We include the proof for convenience of readers.

  Recall that $E\xrightarrow{\iota_E} X_{k, n}$ is the natural embedding, and also denote by $\iota_E$ the induced embedding of moduli spaces of stable maps:
	\[
	\iota_E:	\overline{\mathcal{M}}_{0,2}(E,de)\hookrightarrow\overline{\mathcal{M}}_{0,2}(X_{k, n},de),\quad d\in\bZ_{>0}.
	\]
	Note that a curve in $X_{k,n}$ with degree $de$ must be contained in $E$. So we have
	\[
	\iota_E\Big(\overline M_{0,2}(E,de)\Big)=\overline M_{0,2}(X_{k,n},de).
	\]
	Consider the universal diagram
	\[
	\begin{CD}
		\overline M_{0,3}(E,de)@>ev_3>>E\\
		@V ft_3 VV  \\
		\overline M_{0,2}(E,de)
	\end{CD}
	\]
	and let $R:=R^1(ft_3)_*ev^*_3N_{E|X_{k,n}}$, where $ft_3$ denotes the  natural morphism by forgetting the third marking point. From the construction of virtual fundamental classes, we have
	\[
	(\iota_E)_*\big(\mathbf{e}(R)\cap[\overline M_{0,2}(E,de)]^{\rm virt}\big)=[\overline M_{0,2}(X_{k,n},de)]^{\rm virt},
	\]
	where $\mathbf{e}(R)$ is the Euler class of $R$. So for $\beta_1,\beta_2\in\mathcal B_1$, we see that
	\[
	\<\beta_1,\beta_2\>^{X_{k,n}}_{de}=\int_{[\overline{M}_{0, 2}(E,de)]^{\rm virt}}ev_1^*\iota_E^*\beta_1\cup ev_2^*\iota_E^*\beta_2\cup\mathbf{e}(R).
	\]
	Recall that $E\xrightarrow{\pi_E}G_+$ is the natural projection. For $i=1,2$, consider the morphism $f_i:\overline M_{0,2}(E,de)\to G_+$ defined by the composition
	\[
	f_i:=\pi_E\circ ev_i.
	\]
	We see that $f_1=f_2$, and we let $f:=f_1=f_2$. Then we have
	\[
	\<\beta_1,\beta_2\>^{X_{k,n}}_{de}=\int_{G_+}PD\Big(f_*\big(ev_1^*\iota_E^*\beta_1\cup ev_2^*\iota_E^*\beta_2\cup\mathbf{e}(R)\cap[\overline M_{0,2}(E,de)]^{\rm virt}\big)\Big).
	\]
 Recall that $\xi:=c_1(N_{E|X_{k,n}})\in H^2(E)$. We have
 \begin{align*}
\iota_E^*\sigma_\mu&=\pi_E^*\sigma_\mu,\quad\mu\in\mathcal{P}_{k,n},\\
\iota_E^*(E^i\sigma_\mu)&=\xi^i\cdot\pi_E^*\sigma_\mu,\quad 1\le i\le k-1, \mu\in\mathcal{P}_{k,n-1}.
\end{align*}
 We see that for $i=1,2$, $\iota_E^*\beta_i$ has the form
\[
\iota_E^*\beta_i=\pi_E^*\sigma_{\mu_i}\cup\xi^{m(i)},\quad \mu_i\in\mathcal{P}_{k,n-1},\quad m(i)\in\{0,...,k-1\},
\]
which implies that
\[
ev_i^*\iota_E^*\beta_i=f^*\sigma_{\mu_i}\cup ev_i^*\xi^{m(i)}.
\]
So we use the projection formula to get
\[
\<\beta_1,\beta_2\>^{X_{k,n}}_{de}=\int_{G_+}\sigma_{\mu_1}\cup\sigma_{\mu_2}\cup PD\Big(f_*\big(ev_1^*\xi^{m(1)}\cup ev_2^*\xi^{m(2)}\cup\mathbf{e}(R)\cap[\overline M_{0,2}(E,de)]^{\rm virt}\big)\Big).
\]
Assume that $\<\beta_1,\beta_2\>_{de}\neq0$. Then the above formula gives
\[
\deg\sigma_{\mu_1}+\deg\sigma_{\mu_2}\leq k(n-1-k)\mbox{ and }\sigma_{\mu_1}\cup\sigma_{\mu_2}\neq0.
\]
We use the dimension constraint to get
\begin{align*}
    (k(n-k)-3)+d(k-1)+2={}&\Big(\deg\sigma_{\mu_1}+\deg\sigma_{\mu_2}\Big)+ \Big(m(1)+m(2)\Big)\\
    \leq{}&k(n-1-k)+2(k-1)\\
    ={}& (k(n-k)-3)+d(k-1)+2- (d-1)(k-1).
\end{align*}
Note that $d\in\bZ_{>0}$. Therefore $\<\beta_1,\beta_2\>_{de}\neq0$ only if
\begin{equation}  \label{equ: de}  d=1,\quad\deg\sigma_{\mu_1}+\deg\sigma_{\mu_2}=k(n-1-k),\quad  m(1)=m(2)=k-1.
\end{equation}

From the condition $\sigma_{\mu_1}\cup\sigma_{\mu_2}\neq0$, we only need to consider the Poincar\'e pairing
\[
 {(\beta_1,\beta_2)}=(E^{k-1}\sigma_\mu,E^{k-1}\sigma_{\mu_+^\vee}),\quad\mu\in\mathcal{P}_{k,n-1}.
\]
Since $d=1$, it follows from $H^1(\bP^1,\cO(-1))=0$ that $\mathbf{e}(R)=1$. So we obtain
\[
\<E^{k-1}\sigma_\mu,E^{k-1}\sigma_{\mu_+^\vee}\>^{X_{k,n}}_{e}=\<\xi^{k-1}\pi_E^*\sigma_\mu,\xi^{k-1}\pi_E^*\sigma_{\mu_+^\vee}\>^E_e.
\]
This Gromov-Witten invariant of $E$ is equal to one, since given two distint points in a fiber of $\pi_E$, there is a unique curve of class $e$ passing through them.
\end{proof}

\begin{corollary}\label{cor:devanish}
    For any $\alpha, \beta\in H^*(X_{k, n})$ and $d>1$, $\langle\alpha, \beta\rangle_{d e}=0$.
\end{corollary}
\begin{proof}
   The statement follows directly from
   \eqref{equ: de}.
\end{proof}

\bigskip

\begin{proof}[Proof of Theorem \ref{thm:2ptGW-precise} (c)]

    By \eqref{XknBundle}, we have the bundle structure
     $\pi_2: X_{k,n}=\bP_{G_-}(\mathcal{Q}_{G_-}\oplus\cO_{G_-})\to G_-$.
    Note that  $\ell-e$ is the homology class of a line in a fiber of $\pi_2$. By using similar arguments as in the proof of part (b) of Theorem \ref{thm:2ptGW-precise}, one can prove that the only non-zero, degree-$(\ell-e)$, two-point invariants with insertions in $\mathcal{B}_2$ are
\[
\<E^{n-k}\bar\sigma_\lambda,E^{n-k}\bar\sigma_{\lambda_-^\vee}\>^{X_{k,n}}_{\ell-e}=1,\quad\lambda\in\mathcal{P}_{k-1,n-1}.
\]
\end{proof}

\subsection{Vanishing of   Gromov-Witten invariants of higher degrees}\label{subsec-vanishing}
The two-step flag variety $$F\ell_{k-1, k; n}=\{V_{k-1}\leqslant V_k\leqslant \mathbb{C}^n\mid \dim V_{k-1}=k-1, \dim V_k=k\}=SL(n, \mathbb{C})/P$$
is an $SL(n, \mathbb{C})$-homogeneous variety, where $P$ is a parabolic subgroup of $SL(n,\mathbb{C})$.
Let $T$ (resp. $B$) be the subgroup of $SL(n, \mathbb{C})$ that consists of diagonal (resp. upper triangular) matrices, and denote by $N(T)$ the normalization of $T$ in $SL(n, \mathbb{C})$. Let $\{\alpha_1, \cdots, \alpha_{n-1}\}$ denote the simple roots (in the canonical ordering). The identification of the simple reflection $s_{\alpha_i}$ with the transposition $s_i=(i, i+1)\in S_n$ gives a canonical isomorphism between
the Weyl group $W:=N(T)/T$ and $S_n$. Denote by $\leq$ the Bruhat order on $W$ and let  $len:W\to \mathbb{Z}_{\geq 0}$ denote the standard length function. The longest element $w_0\in W$, in one-line notation, is given by $w_0=[n\cdots 21]$. Set $W^P:=\{w\in S_n\mid w(1)<w(2)<\cdots<w(k-1); w(k+1)<w(k+2)<\cdots<w(k)\}$.
We have the Bruhat decomposition  $F\ell_{k-1, k; n}=\bigsqcup_{u\in W^P}BuP/P,\,\mbox{ where }\, BuP/P\cong \mathbb{C}^{len(u)}$. The Schubert variety $Y(u)$ of $F\ell_{k-1, k; n},$
\begin{equation}
Y(u):=\overline{BuP/P}=\bigsqcup_{v\leq u, v\in W^P} BvP/P,
\end{equation}
 is of dimension $len(u)$ and can be equivalently defined by using dimension conditions on $V_i\cap \Lambda_j$. We remark that the aforementioned Schubert varieties  $X^\lambda$ of $G(k, n)$ are the image of the (opposite) Schubert varieties $w_0Y(w_0u)$ under the natural projection $F\ell_{k-1, k; n}$ for some $u$.
Let $\varpi\in W$ be the permutation $[(n-k+1)(n-k+2)\cdots n12\cdots (n-k)]$ in one-line notation. Then we have
\[Y_{\varpi}=
\{(V_{k-1}\leq V_k)\in F\ell_{k-1, k; n}\mid V_{k-1} \subset\Lambda_{n-1}\}\cong X_{k, n},
\]
where the natural projection $Y_\varpi \to G(k, n)$ is the blowup $\pi: X_{k, n}\to G(k, n)$ by \cite[Proposition 5.4]{BC12}, and the natural projection
$Y_\varpi \to G(k-1, \Lambda_{n-1})$ endows $X_{k, n}=Y_\varpi$ a $\mathbb{P}^{n-k}$-bundle structure.
Whenever $u\leq \varpi$, $X(u):=Y(u)$ is a Schubert subvariety of $Y_\varpi$. Denote by $[X_u]\in H^{2(len(\varpi)-len(u))}(Y_\varpi,\mathbb{Z})$ the Poincar\'e dual of the homology class $[X(u)]\in H_{2len(u)}(Y_\varpi, \mathbb{Z})$.
It follows from the Bruhat decomposition of $Y_\varpi$ that $\{[X_u]\mid u\in W^P, u\leq \varpi\}$ form a basis of $H^*(Y_\varpi, \mathbb{Z})$.
Moreover, $\overline{NE}(Y_\varpi)=\mathbb{Z}_{\geq 0}[X(s_{k-1})]\oplus \mathbb{Z}_{\geq 0}[X(s_{k})]=\mathbb{Z}_{\geq 0}e\oplus \mathbb{Z}_{\geq 0}(\ell-e)$. By the definition of Schubert varieties, we have \begin{equation}
    [X(s_{k-1})]=e,\qquad  [X(s_{k})]=\ell-e. 
\end{equation}
Recall $c_1(X_{k, n})=nH-(k-1)E$ where $\int_eH=0=\int_\ell E$ and $-\int_e E=1=\int_\ell H$. We introduce quantum variables $q_1, q_2$, and denote $\mathbf{q}^{\mathbf{d}}:=q_1^{d_1}q_2^{d_2}$ for $\mathbf{d}=d_1[X(s_{k-1})]+d_2[X(s_k)]$. The monomial $\mathbf{q}^{\mathbf{d}}$ is  equipped with degree $\deg \mathbf{q}^{\mathbf{d}}=d_1\deg q_1+d_2\deg q_2$, where
\begin{equation}
    \deg q_1:=\int_{[X(s_{k-1})]}c_1(X_{k, n})=k-1,\quad \deg q_2=\int_{[X(s_{k})]}c_1(X_{k, n})=n-k+1.
\end{equation}

The following notion of  curve neighborhood   was introduced by Buch, Chaput,
Mihalcea and Perrin \cite{BCMP13} in their study of the quantum $K$-theory ring of  cominuscule Grassmannians, and was further analyzed  for any  homogeneous variety   by Buch and Mihalcea \cite{BM15}.
\begin{definition}
    The curve neighborhood $\Gamma_{\mathbf{d}}(X(u))$ of $X(u)$ of degree $\mathbf{d}\in H_2(X_{k, n}, \mathbb{Z})$ is a reduced subscheme of $X_{k, n}$ defined by
\begin{equation}
\Gamma_{\mathbf{d}}(X(u)):={\rm ev}_2({\rm ev}_1^{-1}(X(u))\subset X_{k, n}.
\end{equation}

\end{definition}

\begin{proposition}\label{prop-key} Let $u\in W^P$ with $u\leq \varpi$ and $\mathbf{d}=d_1[X(s_{k-1})]+d_2[X(s_k)]$. If either {\rm (i)} $d_1\geq 2$, $d_2\geq 0$ and $(d_1, d_2, k)\neq (2, 0, 2)$  or {\rm (ii)} $d_1\geq 0$, $d_2 \geq 2$ hold, then  we have
     $$\Gamma_{\mathbf{d}}(X(u))<\deg \mathbf{q}^{\mathbf{d}}-1 +len(u).$$
  \end{proposition}
In \cite{MiSh}, Mihalcea and Shifler used the curve neighborhood technique to obtain Gromov-Witten invariants of an odd symplectic Grassmannian, which is a smooth Schubert variety in Grassmannian of Lie type $C$. Here we are following their argument for \cite[Theorem 7.1]{MiSh}  to obtain the vanishing of Gromov-Witten invariants of $X_{k, n}$.  We assume the key Proposition \ref{prop-key} first.
\bigskip

\begin{proof}[Proof Theorem \ref{thm:2ptGW-precise} (a)]
Take   $\alpha,\beta\in H^*(X_{k, n})$ and  $\mathbf{d}=d_1[X(s_{k-1})]+d_2[X(s_k)]\in H_2(X_{k, n}, \mathbb{Z})$. We  assume $d_1\geq 0$ and $d_2\geq 0$, otherwise $\langle \alpha, \beta\rangle_{\mathbf{d}}=0$ holds already.  Assume that either $d_1\geq 2$ or $d_2\geq 2$ holds. For any $u\in W^P$ with $u\leq \varpi$, we have
      \begin{align*}
            \langle [X_u], \beta\rangle^{X_{k,n}}_{\mathbf{d}}&=\int_{[\overline{\mathcal{M}}_{0,2}(X_{k,n}, \mathbf{d})]^{\rm vir}} {\rm ev}_1^*[X_u]\cup {\rm ev}_2^*\beta\\
             &=\int_{X_{k,n}} \beta\cup ({\rm ev}_2)_* ({\rm ev}_1^*([X_u])\cap [\overline{\mathcal{M}}_{0,2}(X_{k,n}, \mathbf{d})]^{\rm vir})
         \end{align*}
         The cycle  $({\rm ev}_2)({\rm ev}_1^{-1}(X(u))$ is supported on the curve neighborhood $\Gamma_{\mathbf{d}}(X(u))$, and the
push-forward   $ ({\rm ev}_2)_* ({\rm ev}_1^*([X_u])\cap [\overline{\mathcal{M}}_{0,2}(X_{k,n}, \mathbf{d})]^{\rm vir})$ is non-zero only if the
curve neighborhood  $\Gamma_{\mathbf{d}}(X(u))$ has components of dimension
  $$\mbox{expdim}   \overline{\mathcal{M}}_{0,2}(X_{k,n}, \mathbf{d})-\mbox{codim} X(u)=\deg \mathbf{q}^{\mathbf{d}}-1 +len(u).$$
By Proposition \ref{prop-key}, for $(d_1, d_2, k)\neq (2, 0, 2)$, we have
$\Gamma_{\mathbf{d}}(X(u))<\deg \mathbf{q}^{\mathbf{d}}-1 +len(u)$.
 Hence,  $\langle [X_u], \beta\rangle_{\mathbf{d}}=0$.
  For $(d_1, d_2, k)= (2, 0, 2)$, we also have $\langle [X_u], \beta\rangle_{2e}=0$ by Corollary \ref{cor:devanish}.
  Since $\{[X_u]\}_u$ is a basis, we have $\langle \alpha, \beta\rangle_{\mathbf{d}}=0$.
Hence, $\langle \alpha, \beta\rangle_{\mathbf{d}}\neq 0$ only if $(d_1, d_2)\in \{(0,0), (1,0), (0,1), (1,1)\}$, namely $\mathbf{d}\in \{0, e, \ell-3, \ell\}$.
\end{proof}

\subsubsection{Proof of Proposition \ref{prop-key}}
 Denote $Y=F\ell_{k-1, k; n}$.
Denote by $[Y(s_i)]$ the curve class of $Y(s_i)=X(s_i)$ in $H_2(Y, \mathbb{Z})$, where $i\in\{k-1, k\}$. We also denote $\mathbf{d}:=d_1[Y(s_{k-1})]+d_2[Y(s_k)]$ by abuse notation.
To distinguish with ${\rm ev}_i$, we use $\hat{\rm ev}_i$ for the evaluation maps on $\overline{\mathcal{M}}_{0, 2}(Y, \mathbf{d})$. We will prove  Proposition \ref{prop-key}, mainly by using a careful estimation of the
dimension of the curve neighborhood $\hat\Gamma_{\mathbf{d}}(Y(u)):=\hat{\rm ev}_2(\hat{\rm ev}_1^{-1}(Y(u))\subset Y$   and noting    $\Gamma_{\mathbf{d}}(X(u))\subseteq \hat\Gamma_{\mathbf{d}}(Y(u))$.

The curve neighborhood of a Schubert variety in $Y$, or more generally in a homogeneous variety, is well studied in \cite{BM15}. The Hecke product on $W$ is involved, which, for any $w\in W$ and $1\leq i\leq n-1$, is defined by
\begin{equation}
    w\cdot s_i=\begin{cases}
        ws_i,&\mbox{if } len(ws_i)>len(w),\\
        w,&\mbox{otherwise}.
   \end{cases}
\end{equation}
Here we notice the following facts for permutation group
\begin{equation}\label{SnBruhatorder}
    len(ws_i)>len(w)\Leftrightarrow w(\alpha_i) \in R^+=\bigoplus_{i=1}^{n-1}\mathbb{Z}_{\geq 0}\alpha_i \Leftrightarrow w(i)<w(i+1).
\end{equation}

 For $1\leq i<j\leq n$, the transposition $t_{i,j}=(i, j)$
 is the reflection $s_{\theta_{ij}}$ with respect to the positive root $\theta_{ij}:=\alpha_i+\alpha_{i+1}+\cdots+\alpha_{j-1}\in R^+$.
 It has reduced decompositions
\begin{equation}\label{tijreduced}
    t_{ij}=s_{i}s_{i+1}\cdots s_{j-2} s_{j-1}s_{j-2}\cdots s_{i+1}s_i=s_{j-1}s_{j-2}\cdots s_{i+1} s_{i}s_{i+1}\cdots s_{j-2}s_{j-1}.
\end{equation}
 Denote by $W_P$ the subgroup of $W$ generated by $\{s_1, \cdots, s_{k-2}\}\cup \{s_{k+1}, \cdots, s_{n-1}\}$. As from  \cite[Section 4.2]{BM15}, we let $z_d^P\in W^P$ denote the (unique) minimal length representative of the coset $z_dW_P$, with the element $z_{\mathbf{d}}\in W$   given by the Hecke products
   \begin{equation}
       z_d=\underbrace{t_{1, n}\cdot   \cdots \cdot t_{1, n}}_{d_2} \cdot \underbrace{t_{1, k}\cdot \cdots \cdot t_{1, k}}_{d_1-d_2} \mbox{ if } d_1\geq d_2,  \mbox{ or }   \underbrace{t_{1, n}\cdot   \cdots \cdot t_{1, n}}_{d_1} \cdot \underbrace{t_{k,n}\cdot \cdots \cdot t_{k, n}}_{d_2-d_1} \mbox{ if } d_1<d_2.
   \end{equation}
 \begin{proposition}[\protect{\cite[Theorem 5.1]{BM15}}] \label{prop-cnBM}
 For any $w\in W^P$, we have
 \label{propdeg}
         $$\Gamma_{\mathbf{d}}(Y(w))=Y(w\cdot z_d^P). 
         $$
      \end{proposition}

 \begin{lemma} A reduced decomposition of  $z_d$ is by joining reduced decompositions of $t_{ij}$'s:
   \begin{equation}\label{zdreduce}
       z_d=\begin{cases}
            t_{1, n}t_{2, n-1}   \cdots  t_{d_2, n-d_2+1}{t_{d_2+1, m}  \cdots   t_{d_1, m+d_2-d_1+1}},&\mbox{if } d_1\geq d_2,\\
            t_{1, n}t_{2, n-1}   \cdots  t_{d_1, n-d_1+1}{t_{r, n-d_1}  \cdots   t_{r+d_2-d_1-1, n-d_2+1}},&\mbox{if }  d_1<d_2,
       \end{cases}
   \end{equation}
   where $m=\min\{n-d_2, k\}, r=\max\{d_1+1, k\}$ and $t_{i, j}=t_{ij}$ if $i<j$ or identity if $i\geq j$.
 \end{lemma}

 \begin{proof}
   Assume $d_1\geq d_2$ first. For $1\leq i\leq d_1$, we set $N_i=n-i+1$ if $i\leq d_2$ or $N_i=m+d_2+1-i$  if $i>d_2$.
   Take one reduced decomposition $t_{i, N_i}=s_{\beta_1}\cdots s_{\beta_{2N_i-2i-1}}$ in the first form in
   \eqref{tijreduced}.
   Then for any $1\leq a\leq 2N_i-2i-1$, we have $s_{\beta_a}=s_b$ for some $i\leq b\leq N_i-1$, and
   $i\leq s_{\beta_1}\cdots s_{\beta_{a-1}}(b)<s_{\beta_1}\cdots s_{\beta_{a-1}}(b+1)\leq N_i$.
   Notice $N_{j}>N_{j+1}$ for all $j$. We have
$t_{1, N_1}\cdots t_{i-1, N_{i-1}}s_{\beta_1}\cdots s_{\beta_{a-1}}(b)<s_{\beta_1}\cdots s_{\beta_{a-1}}(b+1)=t_{1, N_1}\cdots t_{i-1, N_{i-1}}s_{\beta_1}\cdots s_{\beta_{a-1}}(b+1)$. Hence,   the expression \eqref{zdreduce} is a reduced decomoposition of $z_d$, by using \eqref{SnBruhatorder}.

    Since  $len(t_{1n}s_1)>len(t_{1n})$, we have the Hecke product $t_{1n}\cdot s_1=t_{1n}$.
   It follows from the above argument for the reduced decomposition of $z_d$ that
   $len(t_{1n}t_{2,n-1})=len(t_{1n})+len(t_{2,n-1})$. Consequently, the Hecke product
   $t_{1,n}\cdot s_2\cdot \cdots \cdot s_{n-2}$
   coincides with the group product
   $t_{1,n} s_2\cdots  s_{n-2}$.
   Note that for any $i< r$,
the expression     $s_{i}s_{i+1}\cdots s_r s_is_{i+1}\cdots s_{r-1}$ is reduced and  $s_{i}s_{i+1}\cdots s_r s_is_{i+1}\cdots s_{r}=s_{i+1}\cdots s_r s_is_{i+1}\cdots s_{r-1}$
    is of smaller length than $len(s_{i}s_{i+1}\cdots s_r s_is_{i+1}\cdots s_{r-1})$. Hence,
    it follows the second reduced decomposition in \eqref{tijreduced} that the Hecke product
    $(t_{1n}s_2s_{3}\cdots s_{n-2})\cdot s_{n-1}=t_{1n}s_2s_{3}\cdots s_{n-2}$
    It follows from the reduced decomposition of
    $t_{1n}t_{2,n-1}$ that the Hecke
    products $ (t_{1n}s_2s_{3}\cdots s_{n-2})\cdot s_{n-3}\cdot \cdots \cdot s_2$ coincides with the group product $t_{1n}t_{2, n-1}$. Since $t_{1n}t_{2,n-1}(1)=n>n-1=t_{1n}t_{2,n-1}(2)$, by \eqref{SnBruhatorder} we have Hecke product $(t_{1n}t_{2,n-1})\cdot s_1= t_{1n}t_{2,n-1}$. Hence, the Hecke product  $t_{1n}\cdot t_{1n}$ coincides with the group product
   $t_{1n}t_{2,n-1}$. By induction, we conclude that $z_d$ is given by the group product of
   the transpositions in \eqref{zdreduce}.

   The arguments for $d_1<d_2$ are similar.
 \end{proof}


 \begin{lemma}\label{lenzdp}
    Let $\mathbf{d}=d_1[Y(s_{k-1})]+d_2[Y(s_k)]$ where either ($d_1\geq 2$ and $d_2\geq 0$) or ($d_1\geq 2$ and $d_2\geq 0$) hold. We have
     $$len(z_{\mathbf{d}}^P)-\deg q_1^{d_1}q_2^{d_2}<\begin{cases} 0,&\mbox{if } (d_1, d_2)=(2, 1) \mbox{ or } (d_1, d_2, k)=(2, 0, 2),\\
     -1,&\mbox{otherwise}.
     \end{cases}$$
     \end{lemma}
 \begin{proof}
 Recall  $\deg  q_1^{d_1}q_2^{d_2}=d_1(k-1)+d_2(n-k+1)
 $. We discuss the case $d_1\geq d_2$ first.

 Note $z_d^P\in W^P$ and the longest element in $W^P$ is of length equal to $\dim F\ell_{k-1, k; n}=k(n-k)+k-1$. If $d_2\geq {n\over 2}$,
 then $\deg  q_1^{d_1}q_2^{d_2}=(d_1-d_2)(k-1)+d_2 n\geq {n^2\over 2}> k(n-k)+k\geq len(z_d^P)+1
 $ (where $2\leq k<n$).
If $d_2\geq k$, then $\deg  q_1^{d_1}q_2^{d_2}=(d_1-d_2)(k-1)+d_2 n\geq kn=k(n-k)+k+k^2-k\geq len(z_d^P)+1$.

 Now assume $d_2<\min\{{n\over 2}, k\}$. Since $W_P$ is generated by $\{s_i\}_{i< k-1 \mbox{ or } i>k}$, it follows from \eqref{zdreduce} that
  $z_dW_P=vW_P$, where  $v$ is taken  one the of the following cases.
  \begin{enumerate}

       \item[1)]   $d_1=d_2 \mbox{ and }k\leq n-d_2$.  Take $v=t_{1, n}t_{2, n-1}   \cdots  t_{d_2, n-d_2+1}$.
  Hence, $v=z_d^P\hat v$ for a unique $\hat v\in W_P$, and we have
 $$len(z_{d}^P)=len(v)-len(\hat v)=len(v)-\sharp
 \{(a, b)\mid  a<b\leq k-1 \mbox{ or } k+2\leq a<b ;  v(a)>v(b)\}.$$
Clearly, for $d_2<b\leq k-1$, we have
 $v(d_2)=t_{1n}\cdots t_{d_2-1, n-d_2+2}(n-d_2+1)=n-d_2+1> k-1=v(k-1)$.
  For $k+2\leq a<n-d_2+1$, we have
 $v(a)=a\geq k+2>d_2=t_{1n}\cdots t_{d_2-1, n-d_2+2}(d_2)=v(n-d_2+1)$.
By similar calculations, we see that
$inv_1:=\{(a, b)\mid 1\leq a\leq d_2,  a<b\leq k-1\}\cup\{(a, b)\mid k+2\leq a<b, n-d_2+1\leq b\leq n\}$ all satisfy $v(a)>v(b)$.
Hence,
 \begin{align*}
    len(z_d^P)\leq len(v)-\sharp inv_1
 &= \sum_{i=1}^{d_2}(2n+1-4i)-\sum_{a=1}^{d_2}(k-1-a)-\sum_{b=n-d_2+1}^n (b-k-1)\\
 &=d_2n+d_2-d_2^2\\
 &\leq  \deg q_1^{d_1}q_2^{d_2}-2.
 \end{align*}

  \item[2)]   $d_1=d_2 \mbox{ and }k>n-d_2$.
 Take $v=t_{1, n}t_{2, n-1}   \cdots  t_{n-k+1, k}$.
 Similar to case 1), $inv_2:=\{(a, b)\mid 1\leq a\leq n-k+1,  a<b\leq k-1\}\cup\{(a, b)\mid k+2\leq a<b, k\leq b\leq n\}$ all satisfy $v(a)>v(b)$. Namely, we just replace $d_2$ in case 1) by $n-k+1$.
 Since  $d_2<{n\over 2}$ and the real function $xn+x-x^2$   is increasing for   $x\in (-\infty, {n+1\over 2})$, we have
  $$len(z_d^P)\leq (n-k+1)n+(n-k+1)-(n-k+1)^2\leq d_2n+d_2-d_2^2
   \leq \deg q_1^{d_1}q_2^{d_2}-2.$$
   \item[3)]   $d_1>d_2 \mbox{ and }k> n-d_2$.   Take $v=t_{1, n}t_{2, n-1}   \cdots  t_{n-k+1, k}$. By the same argument in case 2), we have
    \begin{align*}
        len(z_d^P) \leq d_2n+d_2-d_2^2
  &= \deg q_1^{d_1}q_2^{d_2}+(d_2-d_1)(k-1)+d_2-d_2^2\\
  &\leq \begin{cases}
       \deg q_1^{d_1}q_2^{d_2}-1,&\mbox{if } (d_1, d_2, k)=(2, 1, 2),\\
       \deg q_1^{d_1}q_2^{d_2}-2,&\mbox{otherwise}.
  \end{cases}
    \end{align*}
     \item[4)]   $d_1>d_2 \mbox{ and }k\leq n-d_2$.   Take $t_{1, n}t_{2, n-1}   \cdots  t_{d_2, n-d_2+1} t_{d_2+1, k}$. Then
$inv_4:=\{(a, b)\mid 1\leq a\leq d_2,  a<b\leq k-1\}\cup\{(a, b)\mid k+2\leq a<b, n-d_2+1\leq b\leq n\}\cup \{(d_2+1, b)\mid d_2+1<b\leq k-1\}$ all satisfy $v(a)>v(b)$.
 Compared $inv_4$ with $inv_1$, we have
  \begin{align*}
        len(z_d^P) &\leq d_2n+d_2-d_2^2+len(t_{d_2+1, k})-(k-d_2-2)\\
        &=d_2n+d_2-d_2^2+k-d_2-1\\
  &= \deg q_1^{d_1}q_2^{d_2}+(d_2+1-d_1)(k-1)-d_2^2\\
  &\leq \begin{cases}
       \deg q_1^{d_1}q_2^{d_2}-1,&\mbox{if } (d_1, d_2)=(2, 1) \mbox{ or } (d_1, d_2, k)=(2, 0, 2),\\
       \deg q_1^{d_1}q_2^{d_2}-2,&\mbox{otherwise}.
  \end{cases}
    \end{align*}
\end{enumerate}

Now we consider the case $d_1<d_2$. If $d_1\geq {n\over 2}$,
then  $\deg  q_1^{d_1}q_2^{d_2}=(d_2-d_1)(n-k+1)+d_1 n\geq n-k+1+{n^2\over 2}>k(n-k)+k$. Assume $d_1<{n\over 2}$ now.
   \begin{enumerate}
       \item[5)]$d_1<d_2$ and ($k<d_1+1$ or $k\geq n-d_1$). Take  $t_{1, n}t_{2, n-1}   \cdots  t_{d_1, n-d_1+1}$. By the same argument for case 1) (by replacing $d_2$ there with $d_1$), we have  \begin{align*}
        len(z_d^P) &\leq d_1n+d_1-d_1^2
        =\deg q_1^{d_1}q_2^{d_2}-(d_2-d_1)(n-k+1)+d_1-d_1^2\leq \deg q_1^{d_1}q_2^{d_2}-2.
        \end{align*}
     \item[6)]$d_1<d_2$ and $d_1+1\leq k<n-d_1$.  Take  $t_{1, n}t_{2, n-1}   \cdots  t_{d_1, n-d_1+1}t_{k, n-d_1}$. Similar to case 5), we have
  \begin{align*}
        len(z_d^P) &\leq d_1n+d_1-d_1^2+len(t_{k, n-d_1})-\sharp\{(j, n-d_1)\mid k<j<n-d_1\}\\
        &=d_1n+d_1-d_1^2+n-k-d_1\\
  &= \deg q_1^{d_1}q_2^{d_2}+(d_1-d_2)(n-k+1)+n-k-d_1^2\\
  &\leq    \deg q_1^{d_1}q_2^{d_2}-2.
    \end{align*}
In the last inequality, we notice if  $d_1=0$ then $d_2\geq 2$ by our assumption.
   \end{enumerate}
\end{proof}
\begin{lemma}\label{lem--compX}
 For any $u\in W^P$ with $u\leq \varpi$,
   $\dim \Gamma_{2[X(s_{k-1})]+[X(s_{k})]}(X(u))<\hat\Gamma_{2[Y(s_{k-1})]+[Y(s_{k})]}(Y(u))$.
\end{lemma}
\begin{proof}
  Write $d=2[Y(s_{k-1})]+[Y(s_{k})]$.  Geometrically, $\hat\Gamma_{2[Y(s_{k-1})]+[Y(s_{k})]}(Y(u))$ consists of points $p\in Y$ such that there exit
     a rational curve $C_d$ of degree $d$ passing through $p$ and $Y(u)$. For $\alpha=\alpha_{k-1}+\alpha_k+\cdots+\alpha_{n-1}$ (resp. $\beta=\alpha_1+\alpha_2+\cdots+\alpha_{k-1}$), there is a unique irreducible $T$-stable rational curve $C_\alpha$ of degree $[Y(s_{k-1})]+[Y(s_{k})]$ (resp. $C_\beta$ of degree $[Y(s_{k})]$) jointing the $T$-fix points $t_{k-1, n}.P$ (resp. $t_{1k}.P$) and $1.P$. The union $C_\alpha\cup C_\beta$ is a rational curve of degree $d$ that passing through $1.P\in Y(u)$ and $t_{k-1, n}.P$. Hence,  $t_{k-1, n}.P\in  \hat\Gamma_{d}(Y(u))$.  However, $t_{k-1, n}\not \leq  \varpi$ in the Bruhat order (since the increasing sequence $(1, 2, \cdots, k-2, n)$ obtained by sorting $\{t_{k-1,n}(1), \cdots, t_{k-1, n}(k-1)\}$     is not smaller than
     the increasing sequence $(n-k+1, \cdots, n-1)$ from $\{\varpi(1), \cdots, \varpi(k-1)\}$ in the standard partial order for real vectors). Hence, $t_{k-1, n}.P\notin Y_\varpi=X_{k,n}$. Hence, $\Gamma_{2[X(s_{k-1})]+[X(s_{k})]}(X(u))\subsetneq\hat\Gamma_{2[Y(s_{k-1})]+[Y(s_{k})]}(Y(u))$. Then we are done, by noting  that
    $\hat\Gamma_{2[Y(s_{k-1})]+[Y(s_{k})]}(Y(u))=Y(u\cdot z_d^P)$ is a Schubert variety which is irreducible.
\end{proof}
  \begin{proposition} For any $\mathbf{d}>(1,1)$, we have
     $$\Gamma_{\mathbf{d}}(X(w))<\deg \mathbf{q}^{\mathbf{d}}-1 +\ell(w).$$
  \end{proposition}
   \begin{proof}[Proof of Proposition \ref{prop-key}]
 Note $\Gamma_{\mathbf{d}}(X(u))\subseteq \Gamma_{\mathbf{d}}(Y(u))$. By Proposition \ref{prop-cnBM}, we have
  \begin{align*}
    \dim \Gamma_{\mathbf{d}}(X(u))\leq  \dim\Gamma_{\mathbf{d}}(Y(u))=\dim Y(u\cdot z_d^P)\leq len(u)+len(z_d^P).
  \end{align*}
  Assume $(d_1, d_2)\neq (2, 1)$. Since $(d_1, d_2, k)\neq (2, 0, 2)$,   by Lemma \ref{lenzdp} we have
  $$ \dim \Gamma_{\mathbf{d}}(X(u))\leq len(u)+\deg q^{d_1}_1q_2^{d_2}-2<\deg \mathbf{q}^{\mathbf{d}}+len(u)-1.$$
  If $(d_1, d_2)= (2, 1)$, by Lemma \ref{lem--compX} and Lemma \ref{lenzdp},  we have
    $$ \dim \Gamma_{\mathbf{d}}(X(u))<\dim\Gamma_{\mathbf{d}}(Y(u))\leq   \deg \mathbf{q}^{\mathbf{d}}+len(u)-1.$$
  \end{proof}


\section{ Mirror symmetry for  $X_{2,n}$: A-side}

In this section, we  provide a precise presentation of the quantum cohomology ring $QH^*(X_{2, n})$ on the A-side, and in Section \ref{section-Bside}, we will study the Jacobi ring ${\rm Jac}(f_{\rm tor})$ of the toric superpotential $f_{\rm tor}$ on the B-side. As we will show, both of them are isomorphic to the following algebra (up to an extension on the B-side):
 \begin{equation}
      \mathbb{C}[h, x, q_1, q_2]/(R_{\rm b}(n-1), R_{\sigma}(n-1)-q_2).
 \end{equation}
Here we recall that  $R_{\rm b}(n)$ are  polynomials in $\mathbb{C}[h, x, q_1, q_2]$ recursively defined by
 \begin{equation*}
  R_{\rm b}(0)=1, \quad R_{\rm b}(1)=h,\quad  R_{\rm b}(n)=(h+x)R_{\rm b}(n-1)-(h+q_1)x R_{\rm b}(n-2), \forall n\geq 2,
 \end{equation*}
 and    $R_{\sigma}(n)$ are  polynomials in $\mathbb{C}[h, x, q_1, q_2]$ defined by
  \begin{equation*}
     R_{\sigma}(n)=\sum_{j=0}^{n} x^{n-j} R_{b}(j),\,\,\forall n\in \mathbb{Z}_{\geq 0}.
  \end{equation*}

\subsection{Ring presentation  of $H^*(X_{2,n})$}
As we have discussed in Section \ref{section-Xkngeometry}, there are natural bases $\mathcal{B}_2, \mathcal{B}_1$ (together with their dual bases) of $H^*(X_{2, n})$, arising from  the $\mathbb{P}^{n-2}$-bundle structure and the blowup structure  respectively. Let us re-list them below.
\begin{align*}
\mathcal{B}_2&=\{E^i\bar{H}^j\}_{0\leq i,j\leq n-2}, \mbox{ with dual basis }
	(E^i\bar{H}^j)^\vee=\suml_{a=0}^{n-2-i}E^a\bar{H}^{n-2-i-a}\cdot \bar{H}^{n-2-j}.\\
\mathcal{B}_1&=\{\sigma_{a,b}\}_{n-2\geq a\geq b\geq0}\cup\{E\sigma_{a,b}\}_{n-3\geq a\geq b\geq0}, \quad (\sigma_{a,b})^\vee=\sigma_{n-2-b,n-2-a},\,\,
	(E\sigma_{a,b})^\vee=-E\sigma_{n-3-b,n-3-a}.
\end{align*}
Here   $\bar{H}$ denotes the hyperplane class of $G_-=\bP^{n-2}$. In particular, we have $\bar \sigma_a=\bar H^a$ for all $1\leq a\leq n-2$.
From the $\mathbb{P}^{n-2}$-bundle (resp. blowup) structure, the cohomology ring $H^*(X_{2, n})$ has  natural generators $\{E, \bar H\}$ (resp. $\{E, \sigma_1, \cdots, \sigma_{n-2}\}$). Note that $\sigma_1=\bar H+E$ (see Lemma \ref{lemma-geometricrelationofspecialschubert}). We have the following presentation of $H^*(X_{2 n})$, whose proof will be given at the end of this subsection.
 \begin{proposition}\label{prop-H(X2n)}
  As a $\mathbb{C}$-algebra, the classical cohomology $H^*(X_{2,n})$ is
    generated by $\{\bar H, E\}$ subject to the relations
      \begin{equation*}
          \bar{H}^{n-1}=0,\quad  E^{n-1}+E^{n-2}\bar{H}+\cdots+E\bar{H}^{n-2}=0;
      \end{equation*}
     and it is also generated by $\{E, \sigma_1, \cdots, \sigma_{n-2}\}$ subject to the relations
    \begin{equation}\label{relationHX2n}
       (\sigma_1-E)^{n-1}=0,\quad \sigma_a=\suml_{b=0}^a(-1)^bf_b(a)\sigma_1^{a-b}E^b,\quad 2\leq a\leq n-1.
    \end{equation}
where $\sigma_{n-1}=0$ by convention, and
\(
f_b(a):={a\choose b}+(-1){a-1\choose b-1}+\cdots+(-1)^b{a-b\choose b-b}.
\)
   \end{proposition}


We start with the transformation   between bases $\mathcal{B}_1$ and $\mathcal{B}_2$. In the following lemma, we use $\bar \sigma_a$ (other than $\bar H^a$) to indicate that it holds for general $X_{k, n}$ with the same arguments.

\begin{lemma}\label{lemma-geometricrelationofspecialschubert}
	For any $a\in\{1,2,...,n-2\}$, we have $\bar \sigma_a=\sigma_a-
	E\sigma_{a-1}$.
\end{lemma}
\begin{proof}
The statement holds for $a=1$ by  considering pairing with curve classes.
 	
    Notice the Schubert variety $X^a  =\{V_2\in G(2, n)\mid\dim(V_2\cap \Lambda_{n-1-a})\geq 1\}.
	$
	Consider the complete flag $\Lambda^-_\bullet$ in $\bC^n/\bC_-$ defined by
	$
	\Lambda^-_i:=(\mathbb{C}_-+\Lambda_i)/\bC_-$ for all $i$.
	Then the Schubert class $\bar\sigma_a\in H^*(X_{2,n})$ is the Poincar\'e dual of the homology class of  the subvareity
	\begin{align*}
	S_a:={}&\{(V_{1},V_2)\in F\ell_{1, 2; n}\mid  V_{1}\subset\bC_+^{n-1},
	 \quad\dim\Big(\big((V_{1}+\bC_-)/\bC_-\big)\cap \Lambda^-_{n-1-a}\Big)\geq 1\}.
	\end{align*}
	One can verify that
	\begin{align*}
	S_a={}&\{(V_{1},V_2)\in F\ell_{1, 2; n}\mid    V_{1}\subset\bC_+^{n-1},
 \quad\dim\Big((V_{1}+\bC_-)\cap\mathrm{span}\{e_n,e_1,...,e_{n-1-a}\}\Big)\geq 2\}\\
	={}&\{(V_{1},V_2)\in F\ell_{1, 2; n}\mid V_{1}\subset\bC_+^{n-1},
	\quad\dim(V_{1}\cap \Lambda_{n-1-a})\geq 1\}.
	\end{align*}
	Now we want to show that $S_a$ is the strict transform of $X^a$ under the blow-up $\pi$. By the irreducibility of $S_a$ and $X^a$, it suffices to prove $\pi(S_a)=X^a$. Indeed, for $(V_{1},V_2)\in S_a$,  noting $V_{1}\subset V_2$, we have
$ \dim(V_{2}\cap \Lambda_{n-1-a})\geq \dim(V_{1}\cap \Lambda_{n-1-a})\geq 1$. Thus
 $\pi(S_a)\subset X^a$. Since $S_a$ is irreducible, it follows that $\pi(S_a)$ is an irreducible subvariety of $X^a$. Note that $\pi$ is an isomorphism from $S_a\setminus E\neq\emptyset$ onto its image. So $\dim\pi(S_a)=\dim S_a=\dim X^a$. Now it follows from the irreducibility of $X^a$ that $\pi(S_a)=X^a$.

Consider the complete flag $\Lambda^+_\bullet$ in $\bC_+^{n-1}$ defined by
	\(
	\Lambda^+_i:=\Lambda_i,\,\, 1\leq i\leq n-1.
	\)
	We have
	\begin{align*}
	 X^a\cap G_+
	=\{V_2\subset\bC_+^{n-1}\mid\dim(V_2\cap F_{n-1-a})\geq 1\}
	={}&\{V_2\subset\bC_+^{n-1}\mid\dim(V_2\cap F_{(n-1)-1-(a-1)})\geq 1\}.
	\end{align*}
That is, $X^a\cap G_+$ is a Schubert variety in $G_+$ indexed by the special partition $a-1$. We denote its corresponding Schubert class as $\sigma^+_{a-1}\in H^{2a-2}(G_+)$.

By the blow-up formula \cite[Theorem 6.7]{Fu98}, we have
\[
\sigma_a=\bar\sigma_a+(\iota_E)_*\Big(c(\mathcal{Q}_E)\cap\pi_E^* s\big(X^{a}\cap G_+, X^a\big)\Big),
\]
where $s\big(X^{a}\cap G_+, X^a\big)$ is the Segre class of $X^{a}\cap G_+$ in $X^a$. Note that
\[
s\big(X^{a}\cap G_+, X^a\big)=m\sigma^+_{a-1}+  {\mbox{Poincar\'e duals of lower dimensional cycles}},
\]
where $m$ is the multiplicity of $X^a$ along $X^a\cap G_+$. By comparing dimensions of cycles, we get
\[
\sigma_a=\bar\sigma_a+m(\iota_E)_*\Big(\pi_E^*\sigma^+_{a-1}\Big).
\]
Now it suffices to prove $m=1$. To this end, we only need to show that a general point in $X^a\cap G_+$ is a smooth point in $X^a$. Consider $V_2\in X^a\cap G_+$ such that
\(
\dim(V_2\cap \Lambda^+_{(n-1)-1-(a-1)})=1.
\)
Then $V_2$ is a general point in $X^a\cap G_+$. Since $\Lambda^+_i=\Lambda_i$, it follows that \(
\dim(V_2\cap \Lambda_{n-1-a})=1.
\)
Therefore $V_2$ is a smooth point in $X^a$. This finishes the proof of the lemma.
\end{proof}

\begin{corollary}
    \label{lemma-algebraicrelationofspecialschubert}
	For any  $n-2\geq a\geq b\geq 0$, we have   $\sigma_{a,b}=\suml_{i=b}^aE^{i}\bar\sigma_{a+b-i}$.
\end{corollary}
\begin{proof}
The case $b=0$ follows from the proceding lemma. The general case follows from induction on $b$ and Giambelli's formula for $G(2,n)$:
\[
\sigma_{a,b}=\begin{vmatrix}
	\sigma_{a}&\sigma_{a+1}\\
	\sigma_{b-1}&\sigma_b
\end{vmatrix}.
\]
\end{proof}


\begin{lemma}\label{lemma-classicalpowerofEforX2n}
	For  {$1\leq a\leq n-2$}, we have $E^{a+1}=-\sigma_{a1}+E\sigma_a$.
\end{lemma}
\begin{proof}
	We use induction on $a$.
    By Corollary \ref{lemma-algebraicrelationofspecialschubert},  $\sigma_{11}=E\bar{\sigma}_1$. Therefore, $-\sigma_{11}+E\sigma_1=E(-\bar \sigma_1+\sigma_1)=E^2$. That is, the statement holds for $a=1$.
      Now assume that the required equality holds for $a=a_0\geq1$. Then for $a=a_0+1$, we have
	\begin{align*}
		E^{(a_0+1)+1}=E\cdot E^{a_0+1}=E(-\sigma_{a_01}+E\sigma_{a_0})&=-E\sigma_{a_01}+(-\sigma_{11}+E\sigma_1)\sigma_{a_0}\\
		&=-\sigma_{11}\sigma_{a_0}+E(\sigma_1\sigma_{a_0}-\sigma_{a_01}).
	\end{align*}
	Now the required equality follows from Pieri rule  and Giambelli formula.
\end{proof}

 \bigskip

 \begin{proof}[Proof of Proposition \ref{prop-H(X2n)}]
  Since $\bar H$ is the hyperplane of the base $G_-=\mathbb{P}^{n-2}$, we have $\bar{H}^{n-1}=0$.
  Moreover, since $E$ is the relative hyperplane class of the $\bP^{n-k}$-bundle $X_{k,n}=\bP_{G_-}(\mathcal{Q}_{G_-}\oplus\cO)$, it follows that $\suml_{b=0}^{n-2}\bar{H}^bE^{n-1-b}=0$.

On the other hand, for $a\geq0$, we have $\sigma_a=\suml_{i=0}^aE^i\bar{H}^{a-i}$ by Corollary \ref{lemma-algebraicrelationofspecialschubert}. Therefore
\begin{align}\label{eq-sigmaintermsofE}
\sigma_a=\suml_{i=0}^aE^i(\sigma_1-E)^{a-i}=\suml_{b=0}^a(-1)^bf_b(a)\sigma_1^{a-b}E^b,
\end{align}
 \end{proof}

\subsection{Ring presentation of $QH^*(X_{2, n})$}

By Theorem \ref{thm:2ptGW-precise},   the only non-zero, degree-$(\ell-e)$ (resp. degree-$e$), two-point Gromov-Witten invariants with insertions in $\mathcal{B}_2$ (resp. $\mathcal{B}_1$) are
\begin{align}\label{eq-GWinvofX2nforf}
\<E^{n-2}\bar{H}^j,E^{n-2}\bar{H}^{n-2-j}\>_{\ell-e}=1,\quad 0\leq j\leq n-2,
\end{align}
\begin{align}\label{eq-GWinvofX2nfore}
\<E\sigma_{a,b},E\sigma_{n-3-b,n-3-a}\>_e=1,\quad n-3\geq a\geq b\geq0.
\end{align}

By \cite[Proposition 2.2]{SiTi97}, the quantum cohomology $QH^*(X_{2, n})=\big(H^*(X_{2, n})\otimes \mathbb{C}[q_1, q_2], \star\big)$ of $X_{2,n}$ is generated by the same generators of the classical cohomology $H^*(X_{2, n})$, and the quotient ideal for $QH^*(X_{2, n})$ is simply obtained by  quantizing that for $H^*(X_{2, n})$.   Here  $q_1$ (resp. $q_2$) is the quantum variable corresponding to $e$ (resp. $\ell-e$).  The main result of this section is the following theorem, whose proof will be given at the end of this subsection.
\begin{theorem}\label{thm: QHX2n}
     As a $\mathbb{C}[q_1, q_2]$-algebra,  the quantum cohomology $QH^*(X_{2, n})$ is generated by $\{E, \sigma_1, \cdots, \sigma_{n-2}\}$ subject to the relations:
   \begin{align*}
   	(\sigma_1-E)^{\star(n-1)}&=q_1\suml_{k=2}^{n-1}{(-1)^k}{n-1\choose k}\sigma_1^{\star(n-1-k)}\star(E^{\star(k-1)}+E^{\star(k-2)}\star\sigma_1+\cdots+E\star\sigma_{k-2}),\\
      \sigma_a&=  \suml_{b=0}^a(-1)^bf_b(a)\sigma_1^{\star (a-b)}\star E^{\star b}-q_1\suml_{c=0}^{a-2}\sigma_c\star\suml_{b=c+2}^a(-1)^bf_b(a)\sigma_1^{\star (a-b)}\star E^{\star(b-1-c)}\\
      &\qquad -\delta_{a, n-1} q_2\mathbf{1},\qquad \forall 2\leq a \leq n-1.
	 \end{align*}
    Here $\sigma_{n-1}=0$ by convention.  Moreover,  the map defined by  $\sigma_1\mapsto h+x, E\mapsto x, q_1\mapsto q_1$ and $q_2\mapsto q_2$  is an ring isomorphism:
      \begin{equation*}
  QH^*(X_{2, n})\overset{\cong}{\longrightarrow}  \mathbb{C}[h, x, q_1, q_2]/(R_{\rm b}(n-1), R_{\sigma}(n-1)-q_2).
 \end{equation*}

\end{theorem}

The quantized relations of \eqref{relationHX2n} are again homogeneous polynomials of degree up to $n-1$. Since $\deg q_1=1$ and $\deg q_2=n-1$, it follows that no product term $q_1q_2$ can appear    on the right hand  of the quantized relations. To prove the above theorem,  we can first consider  the restriction  $\star_1$ (resp. $\star_2$)  of $\star$ by setting $q_2=0$ (resp. $q_1=0$). In other words, $\star_i$ only remembers $q_i$.


\begin{lemma}\label{lemma-divisorproduct2forX2n}
For any $\alpha \in H^*(X_{2, n})$ and $0\leq a\leq n-1$,	we have
	\[
	\bar{H}\star_2 \alpha=\bar{H}\cup \alpha,\quad \quad
	E^{\star_2a}=\left\{\begin{array}{ll}
		E^a,&\mbox{if }\,\, 0\leq a\leq n-2,\\
		E^{n-1}+q_2\mathbf{1},&\mbox{if }\,\, a=n-1.
	\end{array}\right.
	\]
\end{lemma}
\begin{proof}
The first equality follows from $\bar{H}.(\ell-e)=0$ and the divisor axiom \eqref{divisorAxiom}.
For the second equality, from \eqref{eq-GWinvofX2nforf}, we get
\(
E\star_2 E^a=E\cup E^{a} \mbox{ for } 0\leq a\leq n-3,
\)
and
\[
E\star_2 E^{n-2}=E^{n-1}+\suml_{\beta\in\mathcal{B}_2}\<E,E^{n-2},\beta\>_{\ell-e}q_2\beta^\vee=E^{n-1}+q_2\mathbf{1}.
\]
For the second equality,  we notice that  only the case $\beta=E^{n-2}\bar{H}^{n-2}$ will make   non-zero contribution and that   $(E^{n-2}\bar{H}^{n-2})^\vee=\mathbf{1}$.
\end{proof}

\begin{lemma}\label{lemma-divisorproductforX2n}
	We have
	\begin{align*}
		\sigma_1\star_1\alpha&=\sigma_1\cup \alpha,\quad\forall \alpha\in H^*(X_{2,n}),\\
		E\star_1\sigma_{ab}&=E\cup\sigma_{ab},\quad \forall n-2\geq a\geq b\geq0,\\
		E\star_1E\sigma_{ab}&=E\cup E\sigma_{ab}+q_1E\sigma_{ab},\quad \forall  n-3\geq a\geq b\geq0.
	\end{align*}
\end{lemma}
\begin{proof}
	The first equality follows from $\sigma_1.e=0$ and the divisor axiom \eqref{divisorAxiom}. The second equality follows from \eqref{eq-GWinvofX2nfore}. For the third equality, we have
	\[
		E\star_1E\sigma_{ab}=E\cup E\sigma_{ab}+\suml_{\beta\in\mathcal{B}_1}\<E,E\sigma_{ab},\beta\>_e^{X(2,n)}q_1\beta^\vee.
	\]
	It follows from the divisor axiom and \eqref{eq-GWinvofX2nfore} that only $\beta=E_{n-3-b,n-3-a}$ has non-zero contribution. Now the required equality follows from $(E\sigma_{n-3-b,n-3-a})^\vee=-E\sigma_{ab}$.
\end{proof}

\begin{lemma}\label{lemma:Eprod}
	For  {$1\leq a\leq n-2$}, we have
	\begin{align*}
	E\star_1 E^a&=E^{a+1}+q_1E\sigma_{a-1},\quad
	\bar{H}\star_1E^a=\bar{H}E^a-q_1E\sigma_{a-1}.
	\end{align*}
\end{lemma}
\begin{proof}
	The first identity follows from Lemmas \ref{lemma-classicalpowerofEforX2n} and \ref{lemma-divisorproductforX2n}, and the second one follows from the first one.
\end{proof}
\begin{lemma}\label{lemma-powerofEforX2n}
	For  {$2\leq a\leq n-2$}, we have
	\begin{align*}
	E^{\star_1a}&=E^a+q_1\Big(E^{\star_1(a-1)}+\suml_{i=1}^{a-2}E^{\star_1(a-1-i)}\sigma_i\Big),\\
	\bar{H}\star_1E^{\star_1a}&=\bar{H}E^a+q_1\Big(\bar{H}\star_1E^{\star_1(a-1)}+\suml_{i=1}^{a-2}\bar{H}\star_1E^{\star_1(a-1-i)}\sigma_i-E\sigma_{a-1}\Big).
	\end{align*}
\end{lemma}
\begin{proof}
	This follows from Lemma \ref{lemma:Eprod} and induction on $a$.
\end{proof}

\begin{lemma}\label{lemma:Hjreduction1}
	For  {$0\leq j\leq n-2$}, we have
	\begin{align*}
E\star_1\bar{H}\star_1\bar{H}^j+q_1E\star_1\bar{H}^j&=E\star_1\bar{H}^{j+1}+q_1E\sigma_j.
\end{align*}

\end{lemma}
\begin{proof}
	 The case $j=0$ is clear.  Now assume $j\geq1$. By  Lemma \ref{lemma-geometricrelationofspecialschubert},  $\bar{H}^j=\bar \sigma_j=\sigma_j-E\sigma_{j-1}$. By Lemma \ref{lemma-divisorproductforX2n}, it suffices to show
	$ 	E\star_1\bar{H}\star_1\bar{H}^j-E\star_1\bar{H}^{j+1}=q_1E\star_1E\sigma_{j-1}.$
    Using Lemma \ref{lemma-divisorproductforX2n} again, we have
	\(
	E\star_1\bar{H}^j=E\bar{H}^j-q_1E\sigma_{j-1}.
	\)
	Therefore 	\begin{align*}
	&E\star_1\bar{H}\star_1\bar{H}^j-E\star_1\bar{H}^{j+1}\\
	={}&\bar{H}\star_1(E\star_1\bar{H}^j)-(E\bar{H}^{j+1}-q_1E\sigma_j)\\
	={}&(\sigma_1-E)\star_1(E\bar{H}^j-q_1E\sigma_{j-1})-(E\bar{H}^{j+1}-q_1E\sigma_j)\\
	={}&(E\sigma_1\bar{H}^j-q_1E\sigma_1\sigma_{j-1}-E\star_1E\bar{H}^j+q_1E\star_1E\sigma_{j-1})-(E\bar{H}^{j+1}-q_1E\sigma_j)\\
	={}&(E\sigma_1\bar{H}^j-E\bar{H}^{j+1}-q_1E\sigma_1\sigma_{j-1}+q_1E\sigma_j)-E\star_1E\bar{H}^j+q_1E\star_1E\sigma_{j-1}\\
	={}&E^2\bar{H}^j+q_1E(\sigma_j-\sigma_1\sigma_{j-1})-E\star_1E\bar{H}^j+q_1E\star_1E\sigma_{j-1}.
	\end{align*}
 So it suffices to prove that $E\star_1E\bar{H}^j=E^2\bar{H}^j+q_1E(\sigma_j-\sigma_1\sigma_{j-1})$. Observe that
\[
E\bar{H}^j=E\sigma_j-E^2\sigma_{j-1}=E\sigma_j-(-\sigma_{11}+E\sigma_1)\sigma_{j-1}=E(\sigma_j-\sigma_1\sigma_{j-1})+\sigma_{11}\sigma_{j-1}.
\]
So the expression for $E\star_1E\bar{H}^j$ follows from Lemma \ref{lemma-divisorproductforX2n}. This finishes the proof.
\end{proof}

In the rest of this subsection, we combine $\star_1$ and $\star_2$ to study $\star$.

\begin{lemma}\label{lemma: Hjrecursive}
	For $0\leq j\leq n-3$, we have
	\begin{equation*}
	   \bar{H}^{j+2}=(\bar H+E)\star \bar H^{j+1}-(\bar H+q_1)\star E\star \bar H^j.
	\end{equation*}
\end{lemma}
\begin{proof}
    By counting degrees and Lemma \ref{lemma:Hjreduction1}, we see that
		\[
	E\star\bar{H}\star\bar{H}^j-E\star\bar{H}^{j+1}+q_1E\star\bar{H}^j-q_1E\sigma_j=\left\{\begin{array}{cc}
0,&\mbox{if }0\leq j<n-3,\\
 m q_2,&\mbox{if }j=n-3,
	\end{array}
	\right.
	\]
	 By Lemma \ref{lemma-divisorproduct2forX2n} and \eqref{eq-GWinvofX2nforf},  we have $m=0$.
By divisor axiom, we have
\[
\bar{H}\star\bar{H}^{j+1}=\bar{H}^{j+2}+\suml_{\beta\in\mathcal{B}_1}\<\bar{H},\bar{H}^{j+1},\beta\>_e^{X_{2,n}}q_1\beta^\vee=\bar H^{j+2}+q_1E\sigma_j.
    \]
    Recall   $\bar{H}^{j+1}=\sigma_{j+1}-E\sigma_j$. By \eqref{eq-GWinvofX2nfore},    only $\beta=E\sigma_{n-3-j}$ has non-zero contribution. Then the second equality above follows from $(E\sigma_{n-3-j})^\vee=-E\sigma_j$. The statement follows from the combination of  the above identities.
\end{proof}

 \bigskip

\begin{proof}[Proof Theorem \ref{thm: QHX2n}]
For the first required equality, by Lemma \ref{lemma-powerofEforX2n}, we have
\begin{align*}
    (\sigma_1-E)^{\star_1 a}
    ={}&\suml_{k=0}^{a}(-1)^k{a\choose k}\sigma_1^{\star_1(a-k)}\star_1E^{\star_1k}\\
    ={}&\sigma_1^{\star_1 a}-a\sigma_1^{\star_1(a-1)}\star_1E+\suml_{k=2}^{a}(-1)^k{a\choose k}\sigma_1^{\star_1(a-k)}\star_1\bigg(E^k+q_1\Big(\sum_{i=1}^{k-1} E^{\star_1(i)}\star_1\sigma_{k-i-1}\Big)\bigg)\\
    ={}&\suml_{k=0}^{a}(-1)^k{a\choose k}\sigma_1^{\star_1(a-k)}\star_1E^k+q_1\suml_{k=2}^{a}(-1)^k{a\choose k}\sigma_1^{\star_1(a-k)}\star_1\Big(\sum_{i=1}^{k-1} E^{\star_1(i)}\star_1\sigma_{k-i-1}
    \Big).
\end{align*}
By Lemma \ref{lemma-divisorproductforX2n}, we have
\[
\suml_{k=0}^{a}(-1)^k{a\choose k}\sigma_1^{\star_1(a-k)}\star_1E^k=\suml_{k=0}^{a}(-1)^k{a\choose k}\sigma_1^{(a-k)}E^k=(\sigma_1-E)^{a}=\bar{H}^{a}.
\]
Then  by counting degrees, we can write
\begin{align*}
     (\sigma_1-E)^{\star a}
    ={}&\bar{H}^a+q_1\suml_{k=2}^{a}(-1)^k{a\choose k}\sigma_1^{\star(a-k)}\star\Big(\sum_{i=1}^{k-1} E^{\star_1(i)}\star\sigma_{k-i-1}\Big)+\delta_{a,n-1}q_2\alpha,
\end{align*}
for some $\alpha\in H^0(X_{2,n})$. When $a=n-1$, we have
\[q_2\alpha=\bar H^{n-1}+q_2\alpha=(\sigma_1-E)^{\star_2(n-1)}=\bar{H}^{\star_2(n-1)}=0,
\]
where the last equality follows from Lemma \ref{lemma-divisorproduct2forX2n}.

Similarly, we have
\begin{align*}
&\suml_{b=0}^a(-1)^bf_b(a)\sigma_1^{\star_1(a-b)}\star_1E^{\star_1b}\\
={}&f_0(a)\sigma_1^{\star_1a}-f_1(a)\sigma_1^{\star_1(a-1)}\star_1E
+\suml_{b=2}^a(-1)^bf_b(a)\sigma_1^{\star_1(a-b)}\star_1\Big(E^b+q_1\sum_{i=1}^{b-1}E^{\star_1(i)} \star_1  \sigma_{b-i-1}\Big)\\
={}&\suml_{b=0}^a(-1)^bf_b(a)\sigma_1^{\star_1(a-b)}\star_1E^{b}+q_1\suml_{b=2}^a(-1)^bf_b(a)\sigma_1^{\star_1(a-b)}\star_1\Big(\sum_{i=1}^{b-1}E^{\star_1(i)} \star_1  \sigma_{b-i-1}\Big)\\
=&\sigma_a+q_1\suml_{c=0}^{a-2}\sigma_c\star_1\suml_{b=c+2}^a(-1)^bf_b(a)\sigma_1^{\star_1(a-b)}\star_1E^{\star_1(b-1-c)},
\end{align*}
where the last equality follows from  Lemma \ref{lemma-divisorproductforX2n} and Proposition \ref{prop-H(X2n)}.
By counting degrees,
\begin{align*}
 \suml_{b=0}^a(-1)^bf_b(a)\sigma_1^{\star(a-b)}\star E^{\star b}
={}&\sigma_a+q_1\suml_{c=0}^{a-2}\sigma_c\star\suml_{b=c+2}^a(-1)^bf_b(a)\sigma_1^{\star(a-b)}\star E^{\star(b-1-c)}+\delta_{a,n-1}q_2\beta,
\end{align*}
for some $\beta\in H^0(X_{2,n})$. When $a=n-1$, and we consider $\star_2$ to get
\begin{align*}
\suml_{b=0}^{n-1}(-1)^bf_b(n-1)\sigma_1^{\star_2(n-1-b)}\star_2 E^{\star_2 b}
={}&\suml_{b=0}^{n-1}(-1)^bf_b(a)(\bar{H}+E)^{\star_2(n-1-b)}\star_2 E^{\star_2 b}\\
={}&\suml_{b=0}^{n-1}\bar{H}^{\star_2b}\star_2E^{\star_2(n-1-b)}\\
={}&\suml_{b=0}^{n-1}\bar{H}^bE^{n-1-b}+q_2\mathbf{1}\qquad\text{(by Lemma \ref{lemma-divisorproduct2forX2n})}\\
={}&q_2\mathbf{1}\qquad\qquad\qquad\quad\text{(by Proposition \ref{prop-H(X2n)}).}
\end{align*}
This proves the second requied equality.

Under the identification $\sigma\mapsto h+x$ and $E\mapsto x$, we have $\bar H^0=1=R_{\rm b}(0)$ and $\bar H=h=R_{\rm b}(1)$.
By Lemma \ref{lemma: Hjrecursive}. $\bar H^j$ and $R_{\rm b}(j)$ satisfy the same recursive relationship. Hence, $\bar H^j=R_{\rm b}(j)$ for all $0\leq j\leq n-1$.
It follows from   the definition that  $R_{\sigma}(j)$'s can also be recursively determined by $R_{\rm b}(j)$'s by the relation $R_{\sigma}(j)-x R_{\sigma}(j-1)=R_{\rm b}(j)$.
By counting degrees and  Lemmas \ref{lemma-geometricrelationofspecialschubert} and  \ref{lemma-divisorproductforX2n},  we have
 $\sigma_j-E\star \sigma_{j-1}=\sigma_j-E \sigma_{j-1}=\bar H^j$ in $QH^*(X_{2, n})$, and conseqently $\sigma_j=R_{\sigma}(j)$ for $0\leq j \leq n-2$.
 Then $R_{\sigma}(n-1)=R_{\rm b}(n-1)+x R_{\sigma}(n-2)=\bar H^{n-1}+E\star \sigma_{n-2}=\bar H^{n-1}+E  \sigma_{n-2}+q_2=\sigma_{n-1}+q_2$.
\end{proof}

\begin{example}\label{exa:X23}
  The case  $X_{2, 3}$ is the blowup of $\mathbb{P}^2$ at one point.
   $R_{\rm b}(2)=h^2-q_1x, R_\sigma(1)=h+x, R_\sigma(2)= R_{\rm b}(2)+ h x   + x^2$. Therefore,
    $QH^*(X_{2, 3})=\mathbb{C}[h, x, q_1, q_2]/I$ with the ideal
    $$I= (R_{\rm b}(2),\,\,  R_\sigma(2)-q_2)
         =(h^2-q_1x, \,\,  hx+x^2-q_2),$$
  coinciding with \cite[Example 7.3]{CM95}.
  \end{example}

\section{ Mirror symmetry for $X_{2, n}$: B-side}
\label{section-Bside}
\subsection{Toric Landau-Ginzburg model} Denote by $([p_i]_i,[p_{jk}]_{j<k}, t)$ the coordinates of  $\mathbb{P}(\wedge^1\mathbb{C}^n)\times \mathbb{P}(\wedge^2\mathbb{C}^n)\times \mathbb{C}$ and  consider the  subvariety  $\mathcal{X}$  defined by
 $$p_{i}p_{jk}-p_{j}p_{ik}+t^{9k-9j}p_kp_{ij}=0,\qquad \forall 1\leq i<j<k\leq n.$$
It is a flat family over $\mathbb{C}$ by the natural projection to the last factor, and there is an isomorphism
  $\mathcal{X}|_{t=1}\times \mathbb{C}^*\cong \mathcal{X}|_{\mathbb{C}^*}$ defined by $p_i\mapsto t^{10i}p_i$ and $p_{jk}\mapsto t^{10j+k}p_{jk}$. Here
  $\mathcal{X}|_{t=1}$ is simply the Pl\"ucker embedding of the two-step flag variety $F\ell_{1, 2; n}$.
  We    identify $X_{2, n}$ with the Schubert divisor $\{V_1\leq V_2\leq \mathbb{C}^n\mid V_1\leq \Lambda_{n-1}\}$ in $F\ell_{1, 2; n}$.
  In particular, $\{p_n=0\}\cap \mathcal{X}$ gives a toric degeneration of  $X_{2, n}$ to the toric variety
    $X_\Delta= \{p_n=0\}\cap \mathcal{X}|_{t=0}$. By \cite[Theorem 3.5]{HLLL22},
    $X_\Delta$ is the toric variety associated to the polytope $\Delta\subset \mathbb{R}^{2(n-2)}$ defined by the ladder diagram in Figure \ref{GCpic}.
       \begin{figure}[h]
    \caption{Ladder diagram for $X_{2, n}$ and its dual graph}\label{GCpic}
    \includegraphics[scale=0.6]{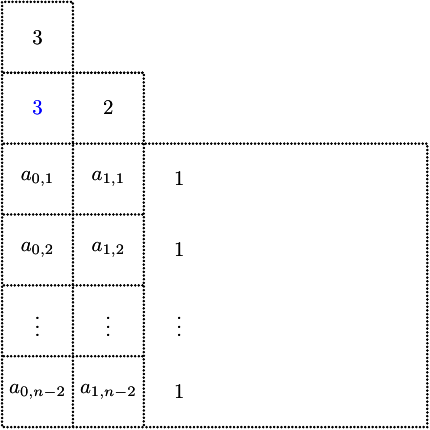}\qquad\qquad
    \includegraphics[scale=0.6]{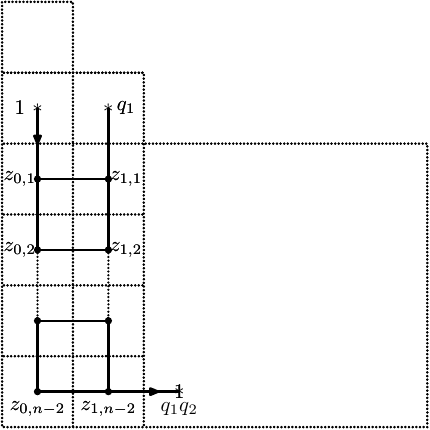}
  \end{figure}

    \noindent That is, after the notation conventions $a_{0,0}=3, a_{1,0}=2$ and $a_{2,j}=1$ for all $j$, we have

       $$\Delta=\{(a_{i,j})_{0\leq i\leq 1\leq j\leq n-2}\in \mathbb{R}^{2(n-2)}\mid a_{i,j-1}\geq a_{i,j}\geq a_{i+1, j},\,\,\forall 0\leq i\leq 1\leq j\leq n-2\}$$
  We view each vertical (resp. horizontal) edge of the dual graph as an arrow with the orientation from top to bottom (resp. from left to right). Then every dot node of the the dual graph becomes the head $h_{a}$ or the tail $t_{a}$ of some arrow $a$. Similar to the case of partial flag varieties in \cite{BCFKS}, the toric superpotential  $f_{\rm tor}$  for $X_{2, n}$  is simply given by $\sum_{a: \rm arrow} {h_a\over t_a}$.
Recalling its  precise  definition in the introduction, we have
  $$f_{\rm tor}: (\mathbb{C}^*)^{2(n-2)}\times (\mathbb{C}^*)^2\to \mathbb{C},$$
   \begin{equation*}
    f_{\rm tor}(z_{i,j}, \mathbf{q}):=z_{0,1}+\sum_{j=2}^{n-2} {z_{0, j}\over z_{0, j-1}}+ {z_{1,1}\over q_1}+\sum_{j=2}^{n-2} {z_{1, j}\over z_{1, j-1}}+\sum_{j=1}^{n-2} {z_{1, j}\over z_{0, j}}+{q_1q_2\over z_{1, n-2}}.
 \end{equation*}
The Jacobi ring ${\rm Jac}(f_{\rm tor})$ is the quotient of the ring  $\mathbb{C}[z_{i,j}^{\pm 1}, q_1^{\pm 1}, q_2^{\pm 1}\mid 0\leq i\leq 1 \leq j\leq n-2]$ by its  Jacobi ideal
\begin{equation}
  \mathcal{J}:=\big(\partial_{z_{i, j}}f_{\rm tor}\mid  0\leq i\leq 1 \leq j\leq n-2\big).
\end{equation}
We freely exchange the notations $z_{ij}=z_{i,j}$ and denote $z_{00}:=1$.
Notice
\begin{equation}\begin{cases}
   z_{0j}^2\partial_{z_{0j}}f_{\rm tor}={z_{0j}^2\over z_{0,j-1}}-{z_{1j}-z_{0,j+1}},&\forall 1\leq j\leq n-3,\\
   z_{0,n-2}^2\partial_{z_{0,n-2}}f_{\rm tor}={z_{0, n-2}^2\over z_{0,n-3}}-{z_{1,n-2}}, \\
   z_{11}^2\partial_{z_{11}}f_{\rm tor}=({1\over q_1}+{1\over z_{01}})z_{11}^2-{z_{12} },& \\
z_{1j}^2\partial_{z_{1j}}f_{\rm tor}=({1\over z_{1,j-1}}+{1\over z_{0j}})z_{1j}^2-{z_{1,j+1}},&\forall 2\leq j\leq n-3,\\
 z_{1, n-2}^2\partial_{z_{1,n-2}}f_{\rm tor}=({1\over z_{1,n-3}}+{1\over z_{0,n-2}})z_{1, n-2}^2-{q_1q_{2}  }, & \\
\end{cases}
\end{equation}
where   $z_{ij}$ are units in  $\mathbb{C}[z_{ij}^{\pm 1}, q_1^{\pm 1}, q_2^{\pm 1}\mid 0\leq i\leq 1 \leq j\leq n-2]$.
Hence, we can view $z_{ij}=z_{ij}(z_{01}, z_{11}, q_1)\in  \mathbb{C}[z_{01}^{\pm 1}, z_{11}^{\pm 1}, q_1^{\pm 1}, q_2^{\pm 1}]$
by using the relations from $\partial_{z_{ij}}f_{\rm tor}$ with $(i=0, 1\leq j\leq n-2)$ and $(i=1, 1\leq j\leq n-4)$. Then
 for $n\geq 4$, we have
    \begin{equation*}
       {\rm Jac}(f_{\rm tor})={\mathbb{C}[z_{ij}^{\pm 1}, q_1^{\pm 1}, q_2^{\pm 1}\mid 0\leq i\leq 1 \leq j\leq n-2]\over \mathcal{J}}\cong
      {\mathbb{C}[z_{01}^{\pm 1}, z_{11}^{\pm 1}, q_1^{\pm 1}, q_2^{\pm 1}]\over (z_{1,n-3}^2\partial_{z_{1,n-3}}f_{\rm tor}, \,\,\, z_{1, n-2}^2\partial_{z_{1,n-2}}f_{\rm tor})}
    \end{equation*}
 up to a step-by-step natural extension of $z_{ij}\in \mathbb{C}^*$ to $z_{ij}\in \mathbb{C}$ on the left hand side.
 In the rest, by   ${\rm Jac}(f_{\rm tor})$ we always take such extension and hence mean the right hand one.
For $n=3$ or $4$, we compute $\mathcal{J}$ directly as follows.

\begin{example}\label{exa:X23B}
  For $n=3$, the Jacobi ideal in $\mathbb{C}[z_{01}^{\pm 1}, z_{11}^{\pm 1}, q_1^{\pm 1}, q_{2}^{\pm 1}]$ is given by
   \begin{align*}
      (1 - {z_{11}\over z_{01}^2}, {1\over q_1} + {1\over z_{01}} - {q_1 q_2\over z_{11}^2})
      &=(z_{01}^2-       z_{11},       {z_{11}^2z_{01}\over q_1}  +z_{11}^2-q_1q_2z_{01}  )  \\
      &=(z_{01}^2-       z_{11},       {z_{11}^2z_{01}\over q_1}  +z_{11}z_{01}^2-q_1q_2z_{01}  )  \\
      &=(z_{01}^2-       z_{11},       {z_{11}^2\over q_1^2}  +{z_{11}z_{01}\over q_1}-q_2  ).
   \end{align*}
 After the substitutions $h=z_{01}$ and $x={z_{11}\over q_1}$, the generators coincide with the polynomials   $R_{b}(2)$ and $R_{\sigma}(2)-q_2$ in Example \ref{exa:X23} respectively.
\end{example}
 \begin{example}\label{exa:X24}
  For  $n=4$,   we have $z_{02}=z_{01}^2-z_{11}, z_{12}={z_{02}^2\over z_{01}}$, and
  \begin{align*}
     -z_{11}^2\partial_{z_{11}}f_{\rm tor}&=z_{12}-({1\over q_1}+{1\over z_{01}})z_{11}^2
      =  {z_{02}^2\over z_{01}}-({1\over q_1}+{1\over z_{01}})z_{11}^2
       =  {(z_{02}-z_{11})(z_{02}+z_{11})\over z_{01}}- {z_{11}^2\over q_1},\\
     z_{12}^2\partial_{z_{12}}f_{\rm tor}+q_1q_2&= z_{12}^2({1\over z_{11}}  +{1\over z_{02}})
      =  {z_{12}^2 z_{01}^2 \over z_{11} z_{02}}
     =  {z_{12}  z_{01} z_{02}       \over z_{11}  }
        =  {  z_{01} z_{02}       \over z_{11}  }({1\over q_1}+{1\over z_{01}})z_{11}^2. \\
  \end{align*}
Hence, after the substitutions $h=z_{01}$ and $x={z_{11}\over q_1}$, we have
  \begin{align*}
     -z_{11}^2\partial_{z_{11}}f_{\rm tor} & =  {(z_{01}^2-2z_{11})  z_{01}}- {z_{11}^2\over q_1}=h^3 - 2 h q_1 x - q_1 x^2=R_{\rm b}(3), \\
    z_{12}^2\partial_{z_{12}}f_{\rm tor}&= {  z_{01} z_{02}     z_{11}  \over q_{1}  }+ {z_{02}z_{11} }-q_1q_2 \\
       &= hx(h^2-q_1x)+q_1x(h^2-q_1x)-q_1q_2
    = q_1 (R_{\sigma}(3)-q_2)+(x-q_1) R_{\rm b}(3),
  \end{align*}
  where
    $R_\sigma(3)= R_{\rm b}(3)+R_{\rm b}(2)x+R_{\rm b}(1)x^2+x^3=R_{\rm b}(3)+(h^2-q_1x)x+hx^2+x^3.$

   Hence, we have
   $${\rm Jac}(f_{\rm tor})\cong {\mathbb{C}[z_{01}^{\pm 1}, z_{11}^{\pm 1}, q_1^{\pm 1}, q_{2}^{\pm 1}]\over
                (-z_{11}^2\partial_{z_{11}}f_{\rm tor}, z_{12}^2\partial_{z_{12}}f_{\rm tor})}
                \cong   {\mathbb{C}[h^{\pm 1}, x^{\pm 1}, q_1^{\pm 1}, q_{2}^{\pm 1}]\over
                (R_{\rm b}(3), R_{\sigma}(3)-q_2)}.   $$

      \end{example}

      \begin{theorem}\label{thm:QHJac} Let $n\geq 4$.  Under the identifications $h=z_{01}$ and $x={z_{11}\over q_1}$, we have
  \begin{align*}
     z_{1, n-3}^2\partial_{z_{1,n-3}}f_{\rm tor}&=-R_{\rm b}(n-1),\qquad  z_{1, n-2}^2\partial_{z_{1,n-2}}f_{\rm tor}= q_1 R_{\sigma}(n-1)+(x-q_1) R_{\rm b}(n-1)-q_1q_2.
  \end{align*}
  Consequently, we have ${\rm Jac}(f_{\rm tor})\cong   {\mathbb{C}[h^{\pm 1}, x^{\pm 1}, q_1^{\pm 1}, q_{2}^{\pm 1}]\over
                (R_{\rm b}(n-1), \,\,\,R_{\sigma}(n-1)-q_2)}$.
\end{theorem}

\subsection{Proof of Theorem \ref{thm:QHJac}}

We simply denote by $A$ the following quantity, which occurs a bit often.
\begin{equation}
   A:= {z_{01}z_{12}\over z_{11}}= z_{01}({1\over q_1}+{1\over z_{01}})z_{11}={z_{01}z_{11}\over q_1}+z_{11}=h x+ q_1 x.
\end{equation}
\begin{lemma}\label{lemma:z1j} For any $1\leq j\leq n-3$, we have
   \begin{equation} \label{formula:z1j}
      z_{1, j+1}=  A{z_{1, j}z_{0, j-1}\over z_{0, j}} =A^{j}{z_{11}\over z_{0j}}.
   \end{equation}
\end{lemma}
\begin{proof}
  By definition of $A$, the statement holds for $j=1$. By relations from $ z_{0,j-1}^2\partial_{z_{0,j-1}}f_{\rm tor}$ and $z_{1j}^2\partial_{z_{1j}}f_{\rm tor}$,
 we have
  $  {z_{1,j+1}\over z_{1,j}}={z_{1, j-1} +z_{0,j}\over z_{1, j-1} z_{0,j}}\cdot {z_{1, j}}={z_{1,j}\over z_{1, j-1} }\cdot {z_{0, j-1}^2\over z_{0, j-2}z_{0, j}}.$
 Hence, we have $$ {z_{1,j+1}\over z_{1,j}}={z_{12}\over z_{11}}\prod_{i=2}^j  {z_{0, i-1}^2\over z_{0, i-2}z_{0, i}}={z_{12}\over z_{11}}\cdot {z_{01}z_{0, j-1}\over z_{0, j}}=A{z_{0, j-1}\over z_{0, j}}.$$ Thus
$
     {z_{1,j+1}\over z_{1,2}}= \prod_{i=2}^j  {z_{1, i+1}\over z_{1, i}}= \prod_{i=2}^j A{z_{0, i-1}\over z_{0, i}}=A^{j-1}{z_{01}\over z_{0j}}$. We are done by noting $A z_{11}=z_{01}z_{12}$.
\end{proof}
\begin{lemma}\label{lemma-reduction}For $n\geq 5$, we have $z_{1, n-2}^2\partial_{z_{1,n-2}}f_{\rm tor}= A z_{0, n-2}-q_1q_2$ and
   $$ z_{1, n-3}^2\partial_{z_{1,n-3}}f_{\rm tor}=   {A z_{1,n-3}z_{0,n-5}+z_{1,n-3}z_{0,n-3} -z_{0, n-2}z_{0,n-3} \over z_{0, n-4}}.$$
\end{lemma}
\begin{proof}
For $n\geq 5$, we have
 \begin{align*} z_{1, n-2}^2\partial_{z_{1,n-2}}f_{\rm tor}+q_1q_2&={z_{1,n-3}+z_{0,n-2}\over z_{1,n-3}z_{0,n-2}}\cdot z_{1, n-2}^2
                     ={z_{0,n-3}^2\over z_{0, n-4}z_{1,n-3}z_{0,n-2}}\cdot {z_{1, n-2}}\cdot {z_{0, n-2}^2\over z_{0, n-3}}  ={A z_{0,n-2}}.
\end{align*}
\begin{align*}
   z_{1, n-3}^2\partial_{z_{1,n-3}}f_{\rm tor}&=({1\over z_{1,n-4}}+{1\over z_{0,n-3}})z_{1,n-3}^2-{z_{0, n-2}^2\over z_{0,n-3}}  \\
          &=z_{1,n-3}\cdot  {z_{1,n-3}\over z_{1,n-4}}+{(z_{1,n-3}-z_{0, n-2})(z_{1,n-3}+z_{0, n-2})\over z_{0,n-3}}  \\
           &=A z_{1,n-3}\cdot  {z_{0,n-5}\over z_{0,n-4}}+{(z_{1,n-3}-z_{0, n-2})z_{0,n-3} \over z_{0, n-4}}.  \qquad\qquad (\mbox{by Lemma } \ref{lemma:z1j})
           \end{align*}
Hence, the statement follows.
\end{proof}
\begin{proposition}\label{prop-zrelation}Under the identifications $h=z_{01}$ and $x={z_{11}\over q_1}$, the following hold.
 \begin{enumerate}
   \item $z_{0, j}=(h+x)z_{0, j-1}-Az_{0,j-2}$, \quad    $\forall 2\leq j\leq n-2$.
   \item $z_{0, j}=R_{\rm b}(j)=R_\sigma(j)-x R_{\sigma }(j-1)$,\quad $\forall 0\leq j\leq n-2$.
   \item $Az_{0, j-1}+(q_1-x)z_{0, j}=q_1R_{\sigma}(j), \quad \forall  1\leq j\leq n-2$.
 \end{enumerate}
 \end{proposition}
  \begin{proof}
Recall $h=z_{01}$ and $x={z_{11}\over q_1}$. We prove the statements by induction on $j$.

   By direct calculations, we have  $z_{02}=z_{01}^2-z_{11}=h^2-q_1x=(h+x)h-(hx+q_1x)\cdot 1$.
Thus statement (1) holds for $j=2$.
    Assume it holds for $2\leq j-1<n-2$.
  \begin{align*}
     z_{0j}
          &={z_{0,j-1}^2\over z_{0, j-2}}-A{z_{1, j-2} z_{0,j-3}  \over z_{0,j-2} }\qquad\qquad\qquad\qquad\qquad (\mbox{by Lemma } \ref{lemma:z1j})\\
          &={z_{0,j-1}((h+x)z_{0, j-2}-Az_{0, j-3})-z_{1, j-2} z_{0,j-3} A \over z_{0,j-2} }\quad (\mbox{by the induction hypothesis})\\
          &=(h+x)z_{0,j-1}- {Az_{0, j-3}(z_{0,j-1}+z_{1, j-2}) \over z_{0,j-2}  } \\
           &=(h+x)z_{0,j-1}- {Az_{0, j-2}  }\qquad\qquad\qquad\qquad\qquad (\mbox{by } z_{0, j-2}^2\partial_{z_{0, j-2}}f_{\rm tor})
  \end{align*}
  Hence,  we conclude statement (1) by induction.

  Notice that both $z_{0j}$ and $R_{\rm b}(j)$ have the same recursive relations as well as the same initial values $z_{00}=1=R_{\rm b}(0)$ and $z_{01}=h=R_{\rm b}(1)$. Hence, the first equality in statement (2) follows. The second equality follows directly from the definition of $R_\sigma(j)$.

  Statement (3) holds for $j=1$ by noting $R_{\sigma}(1)=h+x$. Assuming it holds for $1\leq j-1< n-2$,  we have
   \begin{align*}
     Az_{0, j-1}+(q_1-x)z_{0, j}
          &=A z_{0,j-1}+q_1(R_{\sigma}(j)-x R_{\sigma}(j-1))-xz_{0, j} \\
          &= A z_{0,j-1}+q_1R_{\sigma}(j)-x (Az_{0, j-2}+(q_1-x)z_{0, j-1})-xz_{0, j}\\
          &=q_1R_{\sigma}(j)+x((h+x)z_{0,j-1}-  Az_{0, j-2} - z_{0,j-2}  ) \\
           &=q_1R_{\sigma}(j).
  \end{align*}
  Here the second equality follows from the induction hypothesis, and the last equality holds by statement (1).
  Hence, statement (3) holds by induction.
  \end{proof}

\begin{proof}[Proof of Theorem \ref{thm:QHJac}]By Example \ref{exa:X24}, the statement holds for $n=4$. Assume  $n\geq 5$ now.

The equality for $z_{1, n-2}^2\partial_{z_{1,n-2}}f_{\rm tor}$ holds directly from Lemma \ref{lemma-reduction} and Proposition \ref{prop-zrelation}.

By Proposition \ref{prop-zrelation} (1), for any $3\leq j\leq n-3$, we have
\begin{align*}
   z_{0, j+1}z_{0, j-2}-z_{0, j}z_{0, j-1}&=
 ((h+x)z_{0, j}-Az_{0, j-1})z_{0, j-2}-z_{0, j}((h+x)z_{0,j-2}-Az_{0, j-3})\\
 &=A(z_{0, j}z_{0, j-3}-z_{0, j-1}z_{0, j-2}).
\end{align*}
Thus by induction, we have
 $$ z_{0, j+1}z_{0, j-2}-z_{0, j}z_{0, j-1}=A^{j-2}(z_{03}z_{00}-z_{02}z_{01})
  = -A^{j-2}(h+x)z_{11}$$
Combining this with Lemma \ref{lemma:z1j} and Proposition \ref{prop-zrelation} (1), we have
 \begin{align*}
   z_{1, n-3}^2\partial_{z_{1,n-3}}f_{\rm tor}&=   {z_{1,n-3}(A z_{0,n-5}+z_{0,n-3}) -z_{0, n-2}z_{0,n-3} \over z_{0, n-4}}\\
    &={z_{1,n-3}(h+x)z_{0,n-4}) -z_{0, n-2}z_{0,n-3} \over z_{0, n-4}}\\
    &={(h+x)A^{n-4}z_{11} -z_{0, n-2}z_{0,n-3} \over z_{0, n-4}}
     ={  -z_{0, n-1}z_{0,n-4} \over z_{0, n-4}}
     =-z_{0, n-1}=-R_{\rm b}(n-1).
 \end{align*}
  Note   $q_1$ is a unit.
Hence, the Jacobi ideal  $\mathcal{J}$ is generated by $\{R_{\rm b}(n-1),\, R_{\sigma}(n-1)-q_2\}$.
\end{proof}


\begin{thebibliography}{99}
	


\bibitem[BC12]{BC12} P. Balmer and B. Calm\`{e}s, \emph{Witt groups of Grassmann varieties},
J. Algebraic Geom.   21 (2012), no. 4, 601--642.

\bibitem[BSV24]{BSV24} A. Bansal, S. Sarkar and S. Vats, \emph{Extremal contractions of projective bundles}, preprint at arXiv: math.AG/2409.05091.

\bibitem[Bay04]{Bay04}A. Bayer, \emph{Semisimple quantum cohomology and blowups}, Int. Math. Res. Not. 2004, no. 40, 2069--2083.






\bibitem[BCFKS00]{BCFKS}V. Batyrev, I. Ciocan-Fontanine, B. Kim and  D. van Straten,\,{\it Mirror symmetry and toric degenerations of partial flag manifolds}, Acta Math. 184 (2000), no. 1, 1-39.


\bibitem[BCW02]{BCW02} L. Bonavero, F. Campana and J.A. Wi\'{s}niewski, \emph{Complex manifolds whose blow-up at a point is Fano}, C. R. Math. Acad. Sci. Paris   334 (2002), no. 6, 463--468.


\bibitem[BCMP13]{BCMP13}A.S. Buch, P.-E. Chaput, L. Mihalcea and N. Perrin, \emph{Finiteness of cominuscule quantum $K$-theory}, Annales Sci. de L'\'Ecole Normale Sup\'erieure
(2013), no. 46.

\bibitem[BM15]{BM15} A.S. Buch, L.C.  Mihalcea, {\it Curve neighborhoods of Schubert varieties}, J. Differential Geom. 99 (2015), no. 2, 255--283.

 \bibitem[Cha20]{Cha20} K. Chan, \,{\it SYZ mirror symmetry for toric varieties},
 Adv. Lect. Math. (ALM), 47
 International Press, Somerville, MA, 2020, 1--32.


\bibitem[ChDu23]{ChDu23} B. Chen, C.-Y. Du, {\it On weighted-blowup formulae of genus zero orbifold Gromov-Witten invariants}, Compos. Math. 159(2023), no. 9, 1833--1871.






\bibitem[CLS22]{CLS}T. Coates, W. Lutz and Q. Shafi, \emph{The abelian/nonabelian correspondence and Gromov-Witten invariants of blow-ups},
Forum Math. Sigma   10 (2022), Paper No. e67, 33 pp.




\bibitem[CM95]{CM95}B. Crauder and R. Miranda, \emph{Quantum cohomology of rational surfaces}, The moduli space of curves (Texel Island, 1994), 33--80,
Progr. Math., 129, Birkh\"auser Boston, Boston, MA, 1995.





\bibitem[EHX97]{EHX} T. Eguchi, K. Hori and C.-S. Xiong,{\it Gravitational quantum cohomology}, Int. J. Mod. Phys. A 12 (1997)
1743-1782.




\bibitem[Fult]{Fu98} W. Fulton, {\it Intersection Theory}, second edition, Ergebnisse der Math. und ihrer Grenzgebiete (3), vol. 2, Springer-Verlag, Berlin, 1998.
\bibitem[FP97]{FuPa} W. Fulton, R. Pandharipande,\,\textit{Notes on stable maps
and quantum cohomology}, Proc. Sympos. Pure Math. 62, Part 2, Amer. Math.
Soc., Providence, RI, 1997.


 \bibitem[GGI16]{GGI16}S. Galkin, V. Golyshev and H. Iritani,\,{\it Gamma classes and quantum cohomology of Fano manifolds: gamma conjectures}, Duke Math. J. 165 (2016), no. 11, 2005--2077.

 \bibitem[GHIKLS24]{GHIKLS} S. Galkin, J. Hu, H. Iritani, H. Ke, C. Li and Z. Su, \,{\it Revisiting Gamma conjecture I: counterexamples and modifications},  arXiv: math.AG/2405.16979.

 \bibitem[GaIr19]{GaIr19}S. Galkin and  H. Iritani, \, {\it Gamma conjecture via mirror symmetry}, in \emph{Primitive Forms and Related Subjects}, Kavli IPMU 2014, 55--115, Mathematical Society of Japan, Tokyo, Japan, 2019.

\bibitem[Gat01]{Gat01}A. Gathmann, \emph{Gromov-Witten invariants of blow-ups }, J. Algebraic Geom. 10(3)(2001), 399--432.


\bibitem[Giv95]{Giv95} A. B. Givental, {\it Homological geometry and mirror symmetry},  Proceedings of the International Congress of Mathematicians (ed. S.D. Chatterji), Birkh\"auser Verlag, Basel, 1995, 472-480.

\bibitem[Giv98]{Giv98} A.B. Givental, {\it A mirror theorem for toric complete intersections}, in: Topological Field Theory,
Primitive Forms and Related Topics, Kyoto, 1996, in: Progress in Mathematics, vol. 160, 1998, pp. 141-175.





\bibitem[GP98]{GP98}L. G\"ottsche and R. Pandharipande, \emph{The quantum cohomology of blow-ups of $\bP^2$ and enumerative geometry},  J. Differential Geom. 48(1998), 61--90.


\bibitem[GK24]{GK24} W. Gu and E. Kalashnikov, \emph{A rim-hook rule for quiver flag varieties},
Selecta Math. (N.S.) 30 (2024), no. 3, Paper No. 51, 30 pp.

\bibitem[GS18]{GS18} W. Gu and E. Sharpe, \emph{A proposal for nonabelian mirrors}, arXiv: hep-th/1806.04678.

\bibitem[Hart]{Hartshorne77} R. Hartshorne, {\it Algebraic geometry}, Springer-Verlag, New York, 1977.








\bibitem[HHKQ18]{HHKQ} W. He, J. Hu, H.-Z. Ke and X. Qi, \emph{Blow-up formulae of high genus Gromov-Witten invariants for threefolds},
Math. Z. 290 (2018), no. 3-4, 857--872.

\bibitem[Hu00]{Hu00} J. Hu,\,{\it Gromov-Witten invariants of blow-ups along points and curves}, Math. Z. 233(2000), 709--739.

\bibitem[Hu01]{Hu01} J. Hu, \emph{Gromov-Witten invariants of blow-ups along surfaces}, Compositio Math. 125 (2001), no. 3, 345--352.


 \bibitem[HLR08]{HLR} J. Hu, T.-J. Li and Y. Ruan, \emph{Birational cobordism invariance of uniruled symplectic manifolds}, Invent. Math. 172(2008). 231--275.



\bibitem[HKLY21]{HKLY21} J. Hu, H.-Z. Ke, C. Li and T. Yang, \emph{Gamma conjecture I for del Pezzo surfaces}, Adv. Math. 386(2021), 107797.

\bibitem[HKLS24]{HKLS24} J. Hu, H.-Z. Ke, C. Li and L. Song, \emph{On the quantum cohomology of blow-ups of four-dimensional quadrics}, Acta Math. Sin. (Engl. Ser.) 40(2024), no.1, 313--328.
\bibitem[HLLL22]{HLLL22}  D. Hwang, H. Lee,  J.-H. Lee and C. Li, \emph{$W$-translated Schubert divisors and transversal intersections},  Sci. China Math.  65 (2022), no. 10, 1997--2018.




\bibitem[Hwa12]{Hwa12} J.-M. Hwang, \emph{Geometry of varieties of minimal rational tangents},
Math. Sci. Res. Inst. Publ., 59
Cambridge University Press, Cambridge, 2012, 197--226.

\bibitem[Iri23]{Iri23}H. Iritani, \emph{Quantum cohomology of blowups},  arXiv: math.AG/2307.13555.

\bibitem[IK23]{IK23}H. Iritani and Y. Koto, \emph{Quantum cohomology of projective bundles},   arXiv: math.AG/2307.03696.


\bibitem[Kal24]{Kal24}E. Kalashnikov, \emph{Laurent polynomial mirrors for quiver flag zero loci},
Adv. Math. 445 (2024), Paper No. 109656, 61 pp.

\bibitem[Ke20]{Ke20}H.-Z. Ke, \emph{Gromov-Witten invariants under blow-ups along  $(-1,-1)$-curves}, Michigan Math. J. 69 (2020), no. 3, 515--531.

\bibitem[KLS]{KLSong} H.-Z. Ke, C. Li,   and J. Song, \emph{Mirror symmetry for smooth Schubert varieties in $F\ell_{1, n-1; n}$}, in preparation.
\bibitem[KM05]{KM05}M. Kogan and E. Miller,\,{\it Toric degeneration of Schubert varieties and Gelfand-Tsetlin polytopes}, Adv. Math. 193 (2005), no. 1, 1--17.




\bibitem[KM08]{KM08} J. Koll{\'a}r and S. Mori, {\it Birational Geometry of Algebraic Varieties}, Cambridge University Press, vol. 134, 2008.
\bibitem[Kon21]{Kon21} M. Kontsevich, \emph{Blow-up equivalence}, Talks at Simons Collaboration on Homological Mirror Symmetry
Annual Meeting 2021, 18-19 November, 2021, https://www.simonsfoundation.org/event/
simons-collaboration-on-homological-mirror-symmetry-annual-meeting-2021/,
2021.

\bibitem[Lai09]{Lai} H.-H. Lai, \emph{Gromov-Witten invariants of blow-ups along submanifolds with convex normal bundles}, Geom. Topology, 13(2009), 1--48.

\bibitem[LW90]{LaWe} V. Lakshmibai and J. Weyman, \emph{Multiplicities of points on a Schubert variety in a minuscule  $G/P$}, Adv. Math. 84 (1990), no. 2, 179--208.







\bibitem[LLW16]{LLW16} Y.-P. Lee, H.-W. Lin and C.-L. Wang, {\em Invariance of quantum rings under ordinary flops I: Quantum corrections and reduction to local models}, Algebr. Geom. 3(2016), no. 5, 578--614.

\bibitem[LRYZ24]{LRYZ24}C. Li, K. Rietsch, M. Yang and C. Zhang, \emph{A Plucker coordinate mirror for partial flag varieties and quantum Schubert calculus},   arXiv: math.AG/2401.15640.

\bibitem[Li21]{Li21} D. Li, \emph{Projective bundles and blowing ups},
Comptes Rendus. Math{\'e}matique 359 (2021), no. 9, 1129--1133.

\bibitem[Macd]{Mac95} I. G. Macdonald, {\em Symmetric Functions and Hall Polynomials}, 2nd ed. Oxford, England: Oxford University Press, 1995.
\bibitem[MP06]{MaPa} D. Maulik and R. Pandharipande, \emph{A topological view of Gromov-Witten theory}, Topology 45(2006), 887--918.

\bibitem[MR20]{MR20} R.J. Marsh and K. Rietsch, \emph{The  $B$-model connection and mirror symmetry for Grassmannians},
Adv. Math. 366 (2020), 107027, 131 pp.

\bibitem[MS19]{MiSh} L.C. Mihalcea,  R.M. Shifler, {\it Equivariant quantum cohomology of the odd symplectic Grassmannian},  Math. Z. 291 (2019), no. 3-4, 1569--1603.


\bibitem[MX24]{MX24} T. Milanov and X. Xia, \emph{Reflection vectors and quantum cohomology of blowups},
SIGMA Symmetry Integrability Geom. Methods Appl.   20 (2024), Paper No. 029, 60 pp.

\bibitem[Mok08]{Mok08} N. Mok, \emph{ Recognizing certain rational homogeneous manifolds of Picard number 1 from their varieties of minimal rational tangents}, AMS/IP Stud. Adv. Math., 42, pt. 1, 2
American Mathematical Society, Providence, RI, 2008, 41--61.

\bibitem[Mum66]{Mum66} D. Mumford, \, {\it Lectures on Curves on an Algebraic Surface}, Annals of Math. Studies, vol. 59, Princeton University Press, 1966.

\bibitem[NNU10]{NNU}T. Nishinou, Y. Nohara and K. Ueda,\,{\it Toric degenerations of Gelfand-Cetlin systems and potential functions}, Adv. Math. 224 (2010), no. 2, 648-706.

\bibitem[RR85]{RR85} S. Ramanan, and A. Ramanathan, \emph{Projective normality of flag varieties and Schubert varieties}, Inventiones mathematicae 79 (1985), no. 2, 217--224.


\bibitem[Ram87]{Ramanathan87} A. Ramanathan, \emph{Equations defining Schubert varieties and Frobenius splittings of diagonals}, Publications Math{\'e}matiques de l'IH{\'E}S 65 (1987), 61--90.





\bibitem[Sei97]{Sei97}P. Seidel,\,{\it $\pi_1$ of symplectic automorphism groups and invertibles in quantum homology rings},
Geom. Funct. Anal. 7 (1997), no. 6, 1046--1095.


\bibitem[ST97]{SiTi97} B. Siebert and G. Tian, \emph{On quantum cohomology rings of Fano manifolds and a formula of Vafa and Intriligator}, Asian J. Math. 1(1997), no. 4, 679--695.


\bibitem[Yan22]{Yang}Z. Yang,\,{\it Gamma conjecture I for blowing up $\mathbb{P}^n$ along $\mathbb{P}^r$},   arxiv: math.AG/2022.04234.

\end{thebibliography}
\end{document}